\documentclass[12pt,a4paper]{article}
\usepackage[margin=29mm]{geometry}
\usepackage{amsmath,amssymb,amsthm,mathtools}
\usepackage{enumitem}
\usepackage{lmodern}
\usepackage{microtype}
\usepackage{xcolor}
\usepackage{hyperref}
\usepackage[nameinlink,capitalize,noabbrev]{cleveref}

\hypersetup{
	colorlinks=true,
	linkcolor=blue,
	filecolor=blue,
	urlcolor=blue,
	citecolor=cyan,
}

\allowdisplaybreaks
\numberwithin{equation}{section}
\setlist[enumerate]{leftmargin=2.2em,itemsep=2pt,topsep=3pt}

\newtheorem{theorem}{Theorem}[section]
\newtheorem{lemma}[theorem]{Lemma}
\newtheorem{proposition}[theorem]{Proposition}
\newtheorem{corollary}[theorem]{Corollary}
\newtheorem{problem}[theorem]{Open Problem}

\theoremstyle{definition}
\newtheorem{definition}[theorem]{Definition}
\theoremstyle{remark}

\definecolor{DarkGreen}{rgb}{0.1, 0.5, 0.3}

\newcommand{\hide}[1]{{\color{DarkGreen}#1}}
\renewcommand{\hide}[1]{}

\newcommand{\V}[1]{\boldsymbol{#1}}
\newcommand{\Z}[1]{\mathbb Z_{#1}}
\newcommand{\me}{\mathrm{e}}
\newcommand{\CSOne}{\sigma_1}
\newcommand{\CSTwo}{\sigma_2}
\newcommand{\n}[2]{|#2_{#1}|}
\newcommand{\F}[2]{\mathcal F_{#1,#2}}

\newcommand{\PP}{m}
\renewcommand{\longrightarrow}{\to}

\title{Tur\'an densities of hypergraph augmentations}

\author{
Oleg Pikhurko\thanks{Mathematics Institute and DIMAP, University of Warwick, Coventry CV4 7AL, UK (Email:
o.pikhurko@warwick.ac.uk)}\and Jiabao Yang\thanks{School of Mathematics, Nanjing University, Nanjing 210093, China (Email: jbyang1215@nju.edu.cn)}\and 
Xiutao Zhu\thanks{School of Mathematics, Nanjing University of Aeronautics and Astronautics, Nanjing 211106, China (Email: zhuxiutao@nuaa.edu.cn)}
}
\date{}

\begin{document}

\maketitle
\vspace{-1.2em}

\begin{abstract}
For a $t$-graph $F$ and an integer $r\ge t$, the \emph{$r$-augmentation} $F(r)$ is the $r$-graph obtained by adding the same set of $r-t$ new vertices to every edge of $F$. We investigate the behaviour of the \emph{Tur\'an density $\pi(F(r))$}, which is the asymptotically maximum edge density of a large $F(r)$-free $r$-graph, as a function of $r\to\infty$.

Let $H_k^r$ be the $r$-augmentation of the complete $(k-1)$-graph on $k$ vertices; equivalently, $H_k^r$ is
the (unique up to isomorphism) $r$-graph with $r+1$ vertices and $k$ edges. This paper determines the Tur\'an density of each $r$-graph $H_k^r$ within a constant factor and estimates the lower bounds on $\pi(H_k^r)$ coming from the circular construction of Sidorenko. In particular, it is shown that if $k/\log r\to\infty$ then $\pi(H_k^r)=(1+o(1))k/r$.

Also, we consider augmentations of graphs (that is, the case $t=2$) and 
prove that $\pi(F(r))\ge(\log r)^{-1-o(1)}/r$ for every 3-chromatic graph $F$, which comes close to the known upper bound $1/r$. Note that, for all other graphs $F$, the function $\pi(F(r))$ is known to be either identically 0 or uniformly bounded away from 0.
\end{abstract}



\section{Introduction}

An \emph{$r$-graph} $H$ with vertex set $V$ is a collection of $r$-subsets of $V$. The {\em Tur\'{a}n number} of an $r$-graph $H$, denoted by $\mathrm{ex}(n,H)$, is the maximum number of edges in an $H$-free $r$-graph on $n$ vertices. 
As it was shown by Katona, Nemetz and Simonovits~\cite{KatonaNemetzSimonovits64} (see Lemma~\ref{end:lem:finite-density}), the limit
\[
\pi(H):=\lim_{n\to\infty} \mathrm{ex}(n,H)/\binom {n}{r},
\]
 exists; we call it the \emph{Tur\'an density of $H$}.

Since the fundamental paper by Tur\'an~\cite{Turan41}, Tur\'an problems have been a central topic in extremal combinatorics. For \emph{graphs} (that is, $2$-graphs), Tur\'an densities are well understood: the classical Erd\H{o}s--Stone--Simonovits theorem \cite{ErdosStone46,ErdosSimonovits66} determines $\pi(H)$ for any given graph $H$ in terms of the chromatic number of $H$. However, determining the Tur\'an density of a given $r$-graph with $r\ge 3$ is very difficult in general; we refer the reader to surveys~\cite{Furedi91,Decaen94,Sidorenko95,Keevash11}. 

In order to gain insights, considerable efforts were put towards understanding the growth of $\pi(F_r)$ as $r\to\infty$ for some natural sequences of forbidden $r$-graphs $F_r$. 

One such (2-parameter) forbidden family is as follows. For $2\le k\le r+1$, let $H_k^r$ be the (unique up to an isomorphism) $r$-graph on $r+1$ vertices with $k$ edges. It is easy to see that $\pi(H_2^r)=0$, so we consider $k\ge 3$ only. While an easy averaging argument shows $\pi(H_k^r)\le (k-1)/(r+1)$, a slightly better bound
\begin{equation}\label{eq:(k-2)/r}
	\pi(H_k^r)\le \frac{k-2}r,\quad \text{for every $r+1\ge k\ge 3$}
\end{equation}
follows from Sidorenko~\cite[Eq.~(7)]{Sidorenko82}; for an easily accessible paper containing a proof of~\eqref{eq:(k-2)/r}, see e.g.\ \cite[Corollary~3.9]{MarkstromThomassen21}.

Various special cases of the Tur\'an problem for $H_k^r$ are of great interest in hypergraph Tur\'an theory. For example, if $k=r+1$ then $H_k^r$ is $K_{r+1}^r$,  where $K_m^r$ denotes the complete $r$-graph on $m$ vertices. While the old theorem of Mantel~\cite{Mantel07} shows that the Tur\'an density of the triangle $K_3^2$ is $1/2$, no density $\pi(K_{r+1}^r)$ with $r\ge 3$ is known. In particular, the famous Tetrahedron Conjecture that $\pi(K_4^3)=5/9$ is still open, despite a \$1000 reward offered by Erd\H os (see~\cite[Page 260]{Furedi91}). The currently best known upper bound $\pi(K_4^3)\le 0.5577...$ was obtained by Jeong, Park, Im, Lee and Yang~\cite{JeongParkImLeeYang26x} using flag algebras.
In the next case $r=4$, a construction of Giraud~\cite{Giraud90} gives $\pi(K_5^4)\ge 11/16$, which is conjectured to be sharp. Here, Markstr\"om~\cite{Markstrom09} proved $\pi(K_5^4)\le 0.7365...$ with some small improvements coming from a rounding trick, see~\cite[Section 2]{Sidorenko21}. Interestingly,  Sidorenko~\cite{Sidorenko95} in his survey mentioned that there was no plausible conjecture about the value of $\pi(K_{r+1}^r)$ for any $r\ge 5$ and this still seems to be the case today. However, the behaviour of $\pi(K_{r+1}^r)$ as $r\to\infty$, studied in e.g.\ \cite{Sidorenko82,Decaen83ac,KimRoush83,TazawaShirakura83,FranklRodl85gc,Sidorenko97,ChungLu99,LuZhao09,Pikhurko25,GaoKuangPikhurkoWang26} is now fairly well understood: we know that 
$$
1-\frac{4}{r+2}\le \pi(K_{r+1}^r)\le 1-\frac1r,\quad\text{for every $r\ge2$}.
$$ 
See \cite{PikhurkoVerbitskyZhukovskii25,GaoKuangPikhurkoWang26} for connections of this problem to covering codes.

Another notable case is $H_3^r$ (that is, the $r$-graph with $r+1$ vertices and 3 edges).  If $r=3$ then we get $K_4^3$ minus an edge, for which  Frankl and F\"uredi~\cite{FranklFuredi84} (and independently Giraud, see~\cite[Page 325]{FranklFuredi84}) showed that $\pi(H_3^3)\ge 2/7$, which is conjectured to be sharp (with the best known upper bound $\pi(H_3^3)\le 0.2868...$ obtained by Falgas-Ravry and Vaughan~\cite{FalgasVaughan13} using flag algebras). Gunderson and Semeraro~\cite{GundersonSemeraro17} proved that $\pi(H_3^4)=1/4$. Sidorenko~\cite{Sidorenko24} presented a geometric construction of $H_k^r$-free $r$-graphs for arbitrary $r$ and $k$; in particular, it gives quite good lower bounds on $\pi(H_3^r)$ for small~$r$. Independently of~\cite{Sidorenko24}, 
Gunderson and Semeraro~\cite[Theorem~1.1]{GundersonSemeraro25} produced different constructions for $r\in\{6,7,8\}$.
\hide{Gunderson and Semeraro~\cite{GundersonSemeraro25} proved $\pi(H_3^6)\ge 9/64$, $\pi(H_3^7)\ge 35/2^{11}$, and $\pi(H_3^8)\ge 315/2^{14}$.}

Let us turn our attention to the case when $r\to\infty$. Frankl and F\"uredi~\cite[Example~3]{FranklFuredi84} observed that $\pi(H_3^r)$ is at least $(1+o(1))/2^{r-1}$ as $r\to\infty$.
Sidorenko~\cite[Corollary 4.6]{Sidorenko24} proved that $\pi(H_3^r)\ge (1.7215-o(1))/r^{2}$ and Clemen~\cite[Theorem~1.2]{Clemen26} improved this to $\pi(H_3^r)= \Omega( r^{-2}\sqrt{\log r})$.

The first main result of this paper determines the order of magnitude of $\pi(H_3^r)$. Since the construction easily adapts to work for $H_k^r$ with any constant $k$, we present this more general result.

\begin{theorem}\label{thm:mainPoisson}
	For a positive integer $h$, define
	\begin{equation*}		c_h:=\sup_{\lambda>0}\lambda\me^{-\lambda}
		\sum_{j=0}^{h-1}\frac{\lambda^j}{j!}.
	\end{equation*}
	Then, for every fixed $k\ge3$, we have
	\begin{equation}\label{eq:poisson-fixed-bound}
		\liminf_{r\to\infty}r\pi(H_k^r)
		\ge c_{\lfloor(k-1)/2\rfloor}.
	\end{equation}
\end{theorem}

In the case $k=3$ of Theorem~\ref{thm:mainPoisson}, we have $c_1=\me^{-1}$ (with the supremum attained at $\lambda:=1$), so we obtain that $\pi(H_3^r)\ge (\me^{-1}-o(1))/r$. Thus we know $\pi(H_3^r)$ within a factor $\me+o(1)$ as $r\to\infty$.


For general $H_k^r$, Sidorenko~\cite{Sidorenko24}  gave a geometric construction proving that 
\begin{equation}\label{eq:GeneralLower}
 \pi(H_k^r)\ge \mathbb E(\xi_{r,k-1}),\quad \text{for all $3\le k\le r+1$},
\end{equation}	
that is, $\pi(H_k^r)$ is at least the expectation of $\xi_{r,k-1}$, where 
for $2\le m\le r$ the random variable $\xi_{r,m}$ is the shortest length of an arc that contains at least $m$ points among $r$ independent uniformly distributed points on the circle of circumference~1. 
For each fixed $k\ge3$, it can be derived from the results of Cressie~\cite{Cressie77}  that $\mathbb{E}(\xi_{r,k-1})=(C_k+o(1))r^{-1-\frac{1}{k-2}}$ with an explicit constant $C_k$; 
by~\eqref{eq:GeneralLower} this is a lower bound on $\pi(H_k^r)$. The value of $\mathbb E(\xi_{r,k-1})$ for growing $k=k(r)$ was not estimated in~\cite{Sidorenko24} in the full range since direct computations become tedious. 

By combining the normalised exponential representation for consecutive arc lengths with a maximal inequality (instead of working with $r$-fold integrals), we prove a lower bound on $\mathbb E(\xi_{r,k-1})$ which translates via~\eqref{eq:GeneralLower} into the following lower bound on $\pi(H_k^r)$.

\begin{theorem}\label{thm:main}
	For all integers $r\ge2$ and $3\le k\le r+1$, it holds that
	\begin{equation*}
		\pi(H_k^r)\ge \frac{k-2}{r}- \frac{4}{r}\left(\sqrt{s \log\frac{\me r}{s}}+\log\frac{\me r}{s}\right),
	\end{equation*}
	where $s:=\min\{k-2,r-k+2\}$.
\end{theorem}

To illustrate the obtained lower bound, we write its direct consequences  in some broad ranges of $k=k(r)$, also combining them with the upper bound from~\eqref{eq:(k-2)/r}.

\begin{corollary}\label{cor:applications}
	Let $k=k(r)$ be an integer with $3\le k\le r+1$.
	\begin{enumerate}[label=\textup{(\roman*)}]
\hide{		\item\label{it:Range0} If $k/\log r\to c\in(0,\infty)$, then
		\begin{equation*}
			1-4\left(c^{-1/2}+c^{-1}\right)+o(1)
			\le \frac{\pi(H_k^r)}{(k-2)/r}\le1.
		\end{equation*}
	}
		\item\label{it:Range1} If $k/\log r\to\infty$, then
		\begin{equation*}
			\pi(H_k^r)=(1+o(1))\frac{k}{r}.
		\end{equation*}
		\item\label{it:Range2} If $k/r\to\alpha\in(0,1)$ and $\beta:=\min\{\alpha,1-\alpha\}$, then
		\begin{equation*}
			0\le\frac{k-2}{r}-\pi(H_k^r)
			\le\frac{4\sqrt{\beta\log(\me /\beta)}+o(1)}{\sqrt r}.
		\end{equation*}
	\end{enumerate}
\end{corollary}

Thus, if $k$ grows faster than $\log r$ then we know the value of  $\pi(H_k^r)$ asymptotically.
Also, the error term can be made more precise for larger $k$, with an example of such an estimate given by Item~\ref{it:Range2} above.

In the remaining range $k=O(\log r)$, one can get bounds that are within a constant factor of each other by taking the unions of $r$-graphs returned by Theorem~\ref{thm:mainPoisson}; specifically, the bounds are as follows.

\begin{corollary}\label{cor:half-optimal}
	Let $k=k(r)$ be an integer with $3\le k\le r+1$. Then for every $\varepsilon>0$ there is $r_0$ such that, for all $r\ge r_0$, 
	\begin{equation}\label{eq:half-optimal}
		\pi(H_k^r)\ge \left(\frac1{2\me}-\varepsilon\right)\frac{k-2}{r}.
	\end{equation}
	Moreover, if $k(r)\to\infty$ as $r\to \infty$, then, for all sufficiently large $r$,~\eqref{eq:half-optimal} can be strengthened to $\pi(H_k^r)\ge(\frac12-\varepsilon)\frac{k}{r}$.
\end{corollary}

Thus, by~\eqref{eq:(k-2)/r}, we know $\pi(H_k^r)$ within a multiplicative factor $2\me+o(1)$ uniformly in~$k\in\{3,\dots,r+1\}$ and this factor can be reduced to $2+o(1)$ if $k\to\infty$. Independently of us, Pickhardt~\cite{Pickhardt26} proved $\pi(H_3^r)=\Theta(1/r)$; combined with the same union argument, this also gives $\pi(H_k^r)\ge ck/r$ for an absolute constant $c>0$. (Also, some weaker versions of Theorems~\ref{thm:mainPoisson} and~\ref{thm:main} appeared previously  in a preprint~\cite{YangZhu26x} by the last two authors; this preprint is now fully superseded by the current paper.)

Next, we turn our attention to another important class of sequences of forbidden $r$-graphs. 
For a fixed $t$-graph $F$ and an integer $r\ge t$, let the \emph{$r$-augmentation} $F(r)$ of $F$ be the $r$-graph obtained by adjoining the same set of $r-t$ new vertices to every edge of $F$. This construction is also known in the literature under other names (\emph{$r$-suspension}, \emph{$(r-t)$-fold enlargement}, \emph{(iterated) cone}, etc). There is a large number of results on the Tur\'an function of these $r$-graphs and we do not attempt to survey them here. A connection to the previous results comes from the observation that
$H_k^r$ is the $r$-augmentation of $K_{k}^{k-1}$. 
The augmentations of general complete $t$-graphs $K_k^{t}$ are called \emph{daisies}, a term introduced by Bollob\'as, Leader and Malvenuto~\cite{BollobasLeaderMalvenuto11} who initiated a systematic study of the Tur\'an densities of daisies.
A breakthrough in this area was recently obtained by Ellis, Ivan and Leader~\cite[Theorem~4]{EllisIvanLeader24} who proved, in particular, that $\pi(K_{4}^2(r))$ is bounded away from 0.
This result and the underlying method have already found many applications, including those in the original paper~\cite{EllisIvanLeader24} and the subsequent applications found by Alon~\cite{Alon24} (to erasure list-decodable codes),  Grebennikov and Marciano~\cite{GrebennikovMarciano25} (to edge-sets in hypercubes avoiding a 10-cycle), van der Pol, Walsh and Wigal~\cite{PolWalshWigal25} (to maximising the number of bases in matroids
), Axenovich and Liu~\cite{AxenovichLiu26} (to visibility in hypercubes), etc.

If $t=2$, that is, $F$ is a graph, then the proof in~\cite{EllisIvanLeader24} shows that $\pi(F(r))$ is bounded away from $0$ not only when $F=K_4^2$ is the 4-clique but also when $F$ is an arbitrary graph of chromatic number at least $4$ (see~\cite[Page~3852]{EllisIvanLeader24}). If $F$ is a bipartite graph then the $r$-graph $F(r)$ is $r$-partite, which implies by the classical result of Erd\H os~\cite{Erdos64} that $\pi(F(r))=0$. If $F$ is $3$-chromatic then e.g.\ the more general result of
Cheng, Gerbner, Hama Karim and Zhou~\cite[Theorem 1.2]{ChengGerbnerHamaKarimZhou26} 
implies that $\pi(F(r))\le \pi(H_3^r)$; thus, by~\eqref{eq:(k-2)/r}, we have $\pi(F(r))\le 1/r$.

The following result shows that this is tight  up to a factor $(\log r)^{1+o(1)}$.

\begin{theorem}\label{end:thm:augmentations}
	For every fixed graph $F$ of chromatic number $3$, it holds that
 		\begin{equation*}\label{eq:three-chromatic}
		\pi(F(r))\ge \frac{1}{r(\log r)^{1+o(1)}},\quad \text{as $r\to\infty$}.
	\end{equation*}
\end{theorem}

When one analyses the proof of Theorem~\ref{end:thm:augmentations} and tries to abstract the properties of the graph $F$ that are needed, the following property of general $t$-graphs naturally appears. We say that a $t$-graph $F$ has \emph{zero kernel} if its edge--vertex incidence matrix $B_F$ has zero kernel (or, equivalently, full column rank)  over $\mathbb R$. It is easy to see (see Lemma~\ref{lem:app-graph-rank}) that every connected non-bipartite graph $F$ has zero kernel. Since every graph of chromatic number at least 3 has such a component, Theorem~\ref{end:thm:augmentations} follows from the following more general result.

\begin{theorem}\label{end:thm:main}
	Let $t\ge 2$ and $F$ be a fixed non-empty $t$-graph such that $B_F$ has zero kernel over~$\mathbb R$.
	Then 
	\begin{equation*}\label{eq:full-rank-lower}
		\pi(F(r))\ge
		\frac{1}{r(\log r)^{1+o(1)}},
		\quad\text{as }r\to\infty.
	\end{equation*}
\end{theorem}

This is sharp up to a $(\log r)^{1+o(1)}$ factor if, for example, the $t$-graph $F$ admits a homomorphism to $K_{t+1}^t$ (that is, admits a vertex colouring with $t+1$ colours in which every edge is rainbow). Indeed, then the $r$-augmentation $F(r)$ admits a homomorphism to $K_{t+1}^t(r)=H_{t+1}^r$. This implies via super-saturation (see e.g.~\cite[Theorem 2.2]{Keevash11}) that
$\pi(F(r))\le 
\pi(H_{t+1}^r)$, which by~\eqref{eq:(k-2)/r} is at most $(t-1)/{r}=O(1/r)$.

\medskip\noindent\textbf{Organization of the paper.} Section~\ref{sec:notation} defines some commonly used terms and notions. 
In Section~\ref{sec:nolog},  we prove Theorem~\ref{thm:mainPoisson} using a probabilistic construction involving finite geometries.
In Section \ref{sec:3}, we describe Sidorenko's circular construction and then prove Theorem \ref{thm:main} and Corollary \ref{cor:applications} (that estimate the lower bounds on $\pi(H_k^r)$ coming from the construction) as well as Corollary~\ref{cor:half-optimal} (that extends the bounds to the remaining range $k=O(\log r)$). 
Theorem~\ref{end:thm:main} is proved in Section~\ref{sec:2} via a different probabilistic construction.
As we already mentioned, Theorem~\ref{end:thm:augmentations} directly follows from Theorem~\ref{end:thm:main} when combined with the easy fact (proved in Lemma~\ref{lem:app-graph-rank}) that every connected non-bipartite graph has the zero kernel property. Finally, Section~\ref{sec:concluding} contains some concluding remarks and open problems.

\hide{
In Section \ref{sec:3}, we prove Theorem \ref{thm:main} and Corollary \ref{cor:applications}.  
Sidorenko's circular construction and the normalized exponential representation reduce the required lower bound to estimating the minimum of cyclic sums of independent exponential random variables.  
After centering these variables, we average over all cyclic starting positions and apply an exponential-moment estimate for ordinary partial sums.  
This gives a uniform estimate for the minimum when $m\le r/2$.  When $m>r/2$, a complement identity reduces the problem to cyclic sums of length $r-m$.  Combining the resulting lower bound with the known upper bound proves the finite estimate in Theorem \ref{thm:main}. We later use Theorem \ref{thm:main} to prove Corollary \ref{cor:applications}.
}

\hide{
Let us give a proof outline of the lower bound for $H_3^r$. We choose integers $q$, $\ell$ and $\PP$ depending on $r$. Given $n$, we define two random checksums $\CSOne$ and $\CSTwo$ on $r$-subsets of $[n]$ with values in $\Z{q}$ and $\Z{\PP}$ such that for every $(\kappa_1,\kappa_2)\in \Z{q}\times\Z{\PP}$ the $r$-graph 
\[
\F{\kappa_1}{\kappa_2}:=\left\{A\in\binom{[n]}r : \CSOne(A)=\kappa_1\mbox{ and } \CSTwo(A)=\kappa_2\right\}
\]
 is $H_3^r$-free deterministically (i.e.\ regardless of all the random choices). Then we can take an outcome when the size of $\F{\kappa_1}{\kappa_2}$ is at least the expected value.
 
The first checksum is easy to define: we take a random function $a:[n]\to \Z{q}$ and let $\CSOne(A):=\sum_{v\in A} a(v)$. Instead, we could have fixed an (almost) balanced $q$-partition $[n]=V_0\cup\ldots\cup V_{q-1}$ with the first checksum then being $\sum_{b=0}^{q-1} b\,|V_b\cap A|$ taken modulo $q$; our decision to take a random function $a$ is just a matter of convenience. This idea is abundant in combinatorics; in particular, it appears in the papers~\cite{KimRoush83,FranklRodl85gc,Sidorenko97} giving lower bounds on $\pi(H_{r+1}^r)$.

The key property is that if we take three $r$-sets $A_1,A_2,A_3$ forming a copy of $H_3^r$ with equal first checksums, then the three base vertices $v_1,v_2,v_3$ have the same colour, denote it by $b\in\Z{q}$. Some elements of the stem $D_0$ may also have the colour $b$ so  we let $D:=\{v\in D_0: a(v)=b\}$ and $j:=|D|+2$. Thus if we let $U=D\cup\{v_1,v_2,v_3\}$ then $U\setminus\{v_i\}$ for $i\in [3]$ form a copy of $H_3^j$ with all vertices having the same colour~$b$.

Suppose that $r/q$ grows at least logarithmically in $r$. Then for a `typical' $r$-set $A$ one expects that each colour class inside $A$ has at most $\ell:=Cr/q$ elements for some constant $C>1$. So we can exclude all `atypical' $r$-sets as possible $r$-edges from consideration.

Now, our remaining task is to exclude every $a$-monochromatic copy of $H_3^j$ with $j\le \ell$  lying within a $\CSOne$-monochromatic copy of $H_3^r$. An important observation is that we can choose parameters so that $\ell$ is much smaller than $r$ (with $\ell=O(\log r)$ in the presented proof); thus an extra factor such as $1/\ell^2$ coming from Sidorenko's lower bound is only $1/r^{o(1)}$.

A key idea is to take, for each $j=2,\ldots,\ell$, a partial colouring $\varphi_j:\binom{[n]}j\to L\cup\{\infty\}$ with each finite colour class, as a $j$-graph, being $H_3^j$-free and take the second checksum to be
\begin{equation}\label{eq:CSTwo}
\CSTwo(A):=\sum_{z\in\Z{q}\atop 
|A\cap a^{-1}(z)|\ge 2} \varphi_{|A\cap a^{-1}(z)|}(A\cap a^{-1}(z)),
\end{equation}
 where $L=\Z{\PP}$ is the additive group of residues modulo $\PP$. For a `typical' $r$-set, $|A\cap a^{-1}(z)|\le\ell$ for every $z\in\Z{q}$. Define the right-hand side of~\eqref{eq:CSTwo} only when all its summands are finite. This definition indeed gives that $\F{\kappa_1}{\kappa_2}$ is always $H_3^r$-free. Indeed, if we have a copy of $H_3^r$ then, 
 in the above notation,
 the equality of the checksums $\CSTwo(A_i)$ for $i\in [3]$ implies that the sets $U\setminus \{v_i\}$ for $i\in [3]$ form an edge-monochromatic copy of $H_3^j$ under $\varphi_j$.

Averaging over the checksum values gives expected edge density equal to $\frac1q\cdot\frac1{|L|}$ times the probability that an $r$-set passes the eligibility tests. While we may potentially hope that $|L|=O(\ell^2)$ suffices since Sidorenko's result gives us $H_3^j$-free $j$-graphs of density $\Omega(1/j^2)$, these graphs are not well suited for covering the whole of  $\binom{[n]}j$ and the number of the required copies seems to grow with $n$. The final idea is that we do not need to cover the whole of $\binom{[n]}j$, only a large fraction of it so that the new enlarged set of those $A\in\binom{[n]}r$ for which $\CSTwo(A)$ is undefined has at most, say, $\frac12\binom nr$ elements. 

The combination of the above ideas indeed works. This is a greatly streamlined version of the initial proof suggested by ChatGPT that had various unnecessary complications.
}

\section{Some notation and conventions}\label{sec:notation}

For a positive integer $n$, we denote $[n]:=\{1,\dots,n\}$ (while $[0]$ denotes the empty set) and let $\Z{n}$ denote the additive group of residues modulo $n$. For a prime power $q$, let $\mathbb F_q$ denote the finite field with $q$ elements.
For integers $n,r\ge 0$, the \emph{falling factorial} is $(n)_r:=\prod_{i=0}^{r-1}(n-i)$. An empty product (resp.\ sum) is assumed to have value $1$ (resp.\ $0$).
When we write $\log$, we mean the natural logarithm (i.e.\ with base $\me=2.718...$).

The \emph{argument} (or \emph{angle}) of a non-zero complex number $z=\rho \me^{\mathrm{i}\theta}\in\mathbb C$ with $(\rho,\theta)\in (0,\infty)\times [0,2\pi)$ is 
$\arg(z):=\theta$; thus its value lies in $[0,2\pi)$.

The \emph{binomial distribution} $\mathrm{Bin}(n,p)$ takes value $i\in\{0,\dots,n\}$ with probability $\binom{n}{i}p^i(1-p)^{n-i}$. 
Probability and expectation are denoted by $\mathbb P$ and $\mathbb E$, respectively.

Let $G$ be an $r$-graph. Its vertex and edge sets are denoted by $V(G)$ and $E(G)$.  The number of edges in $G$ is denoted by $|G|$. We call $G$ \emph{non-empty} if $|G|>0$. 
Denoting the vertices of $G$ by $v_1,\dots,v_n$, 
a \emph{blowup} of $G$ is an $r$-graph $G'$ on the union of some pairwise disjoint (possibly empty) sets $V_1,\dots,V_n$ where, for every edge $\{v_{i_1},\dots,v_{i_r}\}\in G$, we add the complete $r$-partite $r$-graph with parts $V_{i_1},\dots,V_{i_r}$ to the edge set of $G'$.

Let $\mathbf 1$ (resp.\ $\V{0}$) denote the all-$1$ (resp.\ all-$0$) vector; its index set should always be clear from the context. The \emph{kernel} $\{\V{v}: M\V{v}=\V{0}\}$ of a matrix $M$ is denoted by $\ker M$. Unless stated otherwise, the underlying field in linear algebra statements is~$\mathbb R$.

\section{Lower bounds on \texorpdfstring{$H_k^r$}{Hkr} via caps and random maps}\label{sec:poisson}\label{sec:nolog}

In this section, we prove Theorem~\ref{thm:mainPoisson}, using random linear maps and known constructions of cap sets over the finite field $\mathbb F_q$ for some prime power $q=q(r,k)$.

We will prove the following result (that has all parameters explicitly written) and then derive Theorem~\ref{thm:mainPoisson} from it.
 

\begin{theorem}\label{thm:poisson-finite}
	Let $q$ be a prime power and let $N\ge r\ge2$ and $3\le k\le r+1$ be integers. Define $h:=\lfloor (k-1)/2\rfloor$ and
	\begin{eqnarray*}
	p_{r,h}(q)&:=&\left(\prod_{i=2}^{r-1}(1-q^{-i})\right)
	\mathbb P\bigl(\operatorname{Bin}(r,1/q)\le h-1\bigr),\\
 s(q)&:=&\frac{(q-1)(q^2+1)}{q^4}.\end{eqnarray*} 
  Then there is an $H_k^r$-free $r$-graph $G$ on $N$ vertices with at least
	$$
	\binom Nr\,p_{r,h}(q)s(q)(1-q^{-3})^{N-r}
	$$
	edges. Consequently,
	\begin{equation}\label{eq:poisson-exact}
		\pi(H_k^r)\ge
		\frac{(N)_r}{N^r}\,p_{r,h}(q)s(q)(1-q^{-3})^{N-r}.
	\end{equation}
\end{theorem}

\subsection{Construction via caps and random linear maps}

Fix $q,N,r,k$ as in Theorem~\ref{thm:poisson-finite}. Recall that $h=\lfloor(k-1)/2\rfloor$. We will use the {projective geometry $\operatorname{PG}(3,q)$}, whose points are the 1-dimensional subspaces of $\mathbb F_q^4$ (called \emph{projective points}). A \emph{projective line} in $\operatorname{PG}(3,q)$ consists of the projective points in a 2-dimensional subspace of $\mathbb F_q^4$.
We write $[\V{v}]$ for the projective point represented by a non-zero vector $\V{v}\in\mathbb F_q^4$. 

A \emph{cap} is a set of points with no three collinear. We use the following classical result of Bose~\cite[Eq.~(5.60)]{Bose47}.
(For modern texts containing a proof, see e.g.~\cite[Theorem 4.36]{Ball15fgca} or~\cite[Proposition~2.6.6]{BrouwerVanmaldeghem22srg}.)

\begin{lemma}[Bose~\cite{Bose47}]\label{lem:cap}
    For every prime power $q$, there is a cap $\mathcal C_q\subseteq\operatorname{PG}(3,q)$ with $q^2+1$ points.
\end{lemma}

\hide{
\begin{proof}
    Let $K:=\mathbb F_{q^2}$, identify $\mathbb F_q^4$ with $\mathbb F_q\times\mathbb F_q\times K$, and write $\operatorname{Nm}(z):=z^{q+1}$ for the norm from $K$ to $\mathbb F_q$. Put
    \[
    \mathcal C_q:=\{[(1,\operatorname{Nm}(z),z)]:z\in K\}\cup\{[(0,1,0)]\}.
    \]
    These are $q^2+1$ distinct points, all represented by zeros of the quadratic form $Q(a,b,z):=ab-\operatorname{Nm}(z)$. If $\V{u}=(1,\operatorname{Nm}(z),z)$ and $\V{v}=(1,\operatorname{Nm}(w),w)$, then expansion of the norm gives
    \[
    Q(s\V{u}+t\V{v})=st\operatorname{Nm}(z-w)
    \quad(s,t\in\mathbb F_q).
    \]
    For $z\ne w$, the norm on the right is non-zero, so the line through $[\V{u}]$ and $[\V{v}]$ contains no other zero of $Q$. If instead $\V{u}=(0,1,0)$ and $\V{v}=(1,\operatorname{Nm}(z),z)$, the corresponding identity is $Q(s\V{u}+t\V{v})=st$, with the same conclusion. Thus no three points of $\mathcal C_q$ are collinear.
\end{proof}
}

Fix such a cap  $\mathcal C_q\subseteq\operatorname{PG}(3,q)$ and let
\[
S_q:=\{\V{v}\in\mathbb F_q^4\setminus\{\V{0}\}:[\V{v}]\in\mathcal C_q\}
\]
be its \emph{cone}. Each projective point has $q-1$ non-zero representatives, so $|S_q|=(q-1)(q^2+1)$ and $|S_q|/q^4=s(q)$.

Choose independent uniform random linear maps
\[
M:\mathbb F_q^N\longrightarrow\mathbb F_q^{r-1}\quad\text{and}\quad
L:\mathbb F_q^N\longrightarrow\mathbb F_q^4;
\]
equivalently, the entries of their matrices are all independent and uniform in $\mathbb F_q$. We use the following standard fact.

\begin{lemma}\label{lem:uniform}
    Let $V,W$ be finite-dimensional vector spaces over $\mathbb F_q$, and choose a uniform linear map $T:V\to W$. If $\V{v}_1,\ldots,\V{v}_s$ are linearly independent in $V$, then $T\V{v}_1,\ldots,T\V{v}_s$ are independent uniform elements of $W$.
\end{lemma}

\begin{proof}
    Extend the given vectors to a basis of $V$. A linear map is uniquely determined by its images on this basis, and all choices of these images occur with equal probability.
\end{proof}

For $A\subseteq[N]$, let $\mathbb F_q^A\subseteq\mathbb F_q^N$ be the coordinate subspace supported on $A$, and let the \emph{kernel} of $A$ be
\[
P_A:=\ker M\cap\mathbb F_q^A,
\]
 that is, the subspace $P_A\subseteq \mathbb F_q^N$ consists of the vectors that are sent to $\V{0}$ by $M$  and whose coordinates in $[N]\setminus A$ are all $0$.

For an $r$-set $A$, we say that $A$ is \emph{admissible} if $\dim P_A=1$ and a non-zero generator $\V{c}_A$ of $P_A$ has at most $h-1$ zero coordinates indexed by $A$. This condition does not depend on the choice of the generator since we require that $P_A$ is 1-dimensional. Choose $\V{c}_A$ as a function of $M$ alone. When $h=1$, an admissible set is a \emph{circuit}: its kernel $P_A$ is generated by a vector supported on all of $A$.

If $A$ is admissible and $A\subseteq B\subseteq[N]$, then the $(r-1)\times r$ submatrix $M|_{\mathbb F_q^A}$ of $M$, where we take the columns indexed by $A$, has the maximum possible rank $r-1$, so $M|_{\mathbb F_q^B}$ also has rank $r-1$. Thus
\begin{equation}\label{eq:poisson-kernel-dimension}
	\dim P_B=|B|-r+1.
\end{equation}
Define the initial $r$-graph $G_0$ by
\begin{equation*}
	A\in G_0\quad\Longleftrightarrow\quad
	A\text{ is admissible and }L\V{c}_A\in S_q.
\end{equation*}
Call an $(r+1)$-set $U$ \emph{bad} if $L(P_U)$ is a 1-dimensional subspace representing a point of~$\mathcal C_q$. Delete from $G_0$ every edge contained in any bad set, and denote the resulting $r$-graph by~$G$.

Before proving that the obtained graph $G$ is always $H_k^r$-free, let us informally discuss this construction. Its most interesting ingredient is probably the map that sends some $r$-sets $A$ (in fact, a positive fraction of them in our final choice of parameters) to vectors $\V{c}_A$ such that, for every $(r+1)$-set $U$, all resulting vectors are confined  to a subspace (namely, $P_U$) whose dimension is $2$ unless some degeneracy occurs in $M|_{\mathbb F_q^U}$ (which is rare so we can just remove all involved $r$-sets).\footnote{Our literature search returned that the idea of placing circuit vectors supported on a set $U$  in a subspace of dimension $|U|-\operatorname{rank}(M|_{\mathbb F_q^U})$ was used by e.g.\ Freij-Hollanti, Jurrius and Kuznetsova~\cite[Equations~(1)--(2)]{FreijHollantiJurriusKuznetsova23}.}
Thus if we take the image of these vectors under the random map $L$ then those coming from any given $(r+1)$-set $U$ can hit at most two different elements of the cap $\mathcal C_q$ and we just need to control the multiplicity of each preimage. 

Let us give a formal proof of the $H_k^r$-freeness of $G$, which is actually quite short.

\begin{lemma}\label{lem:poisson-free}
	For every outcome of the two random maps $M$ and $L$, the resulting $r$-graph $G$ is $H_k^r$-free.
\end{lemma}

\begin{proof}
	Fix an $(r+1)$-set $U$ containing at least one edge of $G_0$. By~\eqref{eq:poisson-kernel-dimension}, we have $\dim P_U=2$.
	
	First, we show that any fixed 1-dimensional subspace of $P_U$ can be the kernel of at most $h$ admissible $r$-subsets of $U$. Suppose that this subspace is $P_A$, where $A=U\setminus\{u\}$ is admissible. Its generator $\V{c}_A$ has zero coordinate at $u$ and at most $h-1$ zero coordinates indexed by $A$, so it has at most $h$ zero coordinates indexed by $U$. If some $r$-set $A'=U\setminus\{u'\}$ has $P_{A'}=P_A$, then the $u'$-coordinate of $\V{c}_A$ is zero. Hence there are at most $h$ possibilities for $A'$, as claimed.
	
If $\dim L(P_U)=2$, then $L|_{P_U}$ is injective. 
So distinct subspaces $P_A$, where $A$ ranges over all initial edges in $U$,  are mapped to distinct projective points of $\mathcal C_q$. 
Since $\mathcal C_q$ is a cap, there are at most two such subspaces $P_A$. Each accounts for at most $h$ initial edges, so $U$ contains at most $2h$ edges of $G_0$.
	
	The case $\dim L(P_U)=0$ is impossible, since an initial edge $A\subseteq U$ forces a non-zero vector $L\V{c}_A\in L(P_U)\cap S_q$. If $\dim L(P_U)=1$, then the presence of an initial edge implies that the projective point corresponding to $L(P_U)$ belongs to $\mathcal C_q$. Thus $U$ is bad and all initial edges inside $U$ are deleted.

	In every case, $U$ contains at most $2h\le k-1$ edges of $G$, as required.
\end{proof}

\subsection{Determining the expected size of \texorpdfstring{$G$}{G}}

Next, for a fixed $r$-set $A\subseteq [N]$, we write an explicit formula for the probability that $A$ is an edge of~$G$.

We first compute the probability that $A$ is admissible. The $(r-1)\times r$ submatrix $M|_{\mathbb F_q^A}$ of $M$ has rank $r-1$ with probability
\[
\prod_{j=0}^{r-2}\left(1-\frac{q^{j}}{q^{r}}\right)=\prod_{i=2}^{r}(1-q^{-i}),
\]
as follows by exposing its rows successively and requiring linear independence. Conditional on this rank, its kernel is a uniform 1-dimensional subspace of $\mathbb F_q^A$. Indeed, right multiplication by any invertible $r\times r$ matrix preserves the distribution of the submatrix and acts transitively on its possible kernel lines.

A uniform kernel line has the same distribution of zero coordinates as a uniform non-zero vector in $\mathbb F_q^A$, since every line has $q-1$ non-zero vectors. The number of zero coordinates of a uniform vector in $\mathbb F_q^A$ (without the non-zero conditioning) is  distributed as $Z\sim\operatorname{Bin}(r,1/q)$. Since $h-1<r$, the event $Z\le h-1$ excludes the zero vector. Hence
\begin{align*}
	\mathbb P(A\text{ is admissible})
	&=\left(\prod_{i=2}^{r}(1-q^{-i})\right)
	\frac{\mathbb P(Z\le h-1)}{1-q^{-r}}
	=p_{r,h}(q).
\end{align*}

Next, condition on the entire matrix $M$ such that $A$ is admissible. By Lemma~\ref{lem:uniform}, the image $L\V{c}_A$ is uniform in $\mathbb F_q^4$, so it belongs to $S_q$ with probability $|S_q|/q^4$. For each $x\in[N]\setminus A$, choose a vector $\V{u}^x\in P_{A\cup\{x\}}$ with $x$-coordinate equal to 1, making this choice as a function of $M$ alone. Such a vector exists: the coordinate projection from  $P_{A\cup\{x\}}$ (which has dimension 2 by~\eqref{eq:poisson-kernel-dimension}) to $\mathbb F_q^{\{x\}}$ has 1-dimensional kernel $P_A$ so its image is 1-dimensional and contains the whole of $\mathbb F_q^{\{x\}}$.

The vectors
\[
\{\V{c}_A\}\cup\{\V{u}^x:x\in[N]\setminus A\}
\]
are linearly independent: in any linear relation between them, the coordinate at $x\in [N]\setminus A$ is the coefficient of $\V{u}^x$, and once all these coefficients vanish, the coefficient of $\V{c}_A$ vanishes too. Their images under $L$ are therefore independent uniform vectors by Lemma~\ref{lem:uniform}.

Condition further on $L\V{c}_A=\V{y}$ for a fixed $\V{y}\in S_q$. The vectors $L\V{u}^x$ remain independent and uniform in $\mathbb F_q^4$. Since
\[
P_{A\cup\{x\}}=\operatorname{span}\{\V{c}_A,\V{u}^x\},
\]
the set $A\cup\{x\}$ is bad exactly when $L\V{u}^x\in\operatorname{span}\{\V{y}\}$. This has probability $q/q^4=q^{-3}$, independently for each $x\in [N]\setminus A$. These are all the $(r+1)$-sets containing $A$. Consequently,
\begin{equation*}
	\mathbb P(A\in G)=\mathbb P(A\text{ is admissible})\, \frac{|S_q|}{q^4}\, (1-q^{-3})^{N-r}=p_{r,h}(q)s(q)(1-q^{-3})^{N-r}.
\end{equation*}
Summing over all $r$-sets $A$ yields
\begin{equation}\label{eq:poisson-expectation}
	\mathbb E\,|G|=\binom Nr p_{r,h}(q)s(q)(1-q^{-3})^{N-r}.
\end{equation}

\subsection{Deriving  Theorems~\ref{thm:poisson-finite} and~\ref{thm:mainPoisson}}

\begin{proof}[Proof of Theorem~\ref{thm:poisson-finite}]
	By Lemma~\ref{lem:poisson-free}, every outcome of $G$ is $H_k^r$-free. Choose an outcome whose edge count is at least the expectation in~\eqref{eq:poisson-expectation}. This proves the assertion about the number of edges.
	
	Every pair of vertices of $H_k^r$ lies in an edge (since $k\ge 3$). Consequently, any copy of $H_k^r$ in a blowup of $G$ would use every part at most once and would project to a copy of $H_k^r$ in $G$. Thus all blowups of $G$ are also $H_k^r$-free. Replacing each vertex of $G$ by $m$ clones and letting $m\to\infty$ gives
	\[
	\pi(H_k^r)\ge\lim_{m\to\infty}\frac{|G|m^r}{\binom{Nm}r}
	=\frac{r!|G|}{N^r}
	\ge\frac{(N)_r}{N^r}\,p_{r,h}(q)s(q)(1-q^{-3})^{N-r},
	\]
	which proves~\eqref{eq:poisson-exact}.
\end{proof}

Next, we derive Theorem~\ref{thm:mainPoisson} from Theorem~\ref{thm:poisson-finite}. 
We will use the following elementary inequality:
\begin{equation}\label{eq:product-bound}
	\prod_{i=1}^m(1-a_i)\ge1-\sum_{i=1}^m a_i,
	\quad\text{for any }a_1,\ldots,a_m\in[0,1],
\end{equation}
which follows by induction on $m$. 

\hide{
\subsection{The Poisson constants}

Recall also the constants $c_h$ defined in Theorem~\ref{thm:mainPoisson}. We first record bounds on them.

\begin{lemma}\label{lem:poisson-constants}
	For every integer $h\ge1$, we have
	\begin{equation}\label{eq:poisson-constant-bounds}
		h-\sqrt{2h\log h}-1\le c_h\le h.
	\end{equation}
	Also, $c_1=\me^{-1}$ and $c_h/h\to1$ as $h\to\infty$.
\end{lemma}

\begin{proof}
	If $Z\sim\operatorname{Pois}(\lambda)$, then direct summation gives
	\[
	\lambda\mathbb P(Z\le h-1)
	=\mathbb E\bigl[Z\,\mathbf 1_{\{Z\le h\}}\bigr]\le h.
	\]
	Taking the supremum over $\lambda>0$ proves the upper bound.
	
	For the lower bound, there is nothing to prove when $h-\sqrt{2h\log h}-1\le0$. Otherwise, put $t:=\sqrt{2h\log h}$ and $\lambda:=h-t>0$; in this case $h>1$. The defining series of the Poisson distribution gives
	\[
	\mathbb E\,\me^{uZ}=\exp\bigl(\lambda(\me^u-1)\bigr).
	\]
	Markov's inequality with $u:=\log(h/\lambda)>0$ yields
	\begin{align*}
		\mathbb P(Z\ge h)
		&\le\exp\bigl(-h\log(h/\lambda)+h-\lambda\bigr)
		\le\exp\left(-\frac{(h-\lambda)^2}{2h}\right)=\frac1h.
	\end{align*}
	For the second inequality, write $x:=(h-\lambda)/h\in(0,1)$ and use
	\[
	-\log(1-x)-x=\int_0^x\frac{u}{1-u}\,du\ge\frac{x^2}{2}.
	\]
	It follows that
	\[
	c_h\ge\lambda\mathbb P(Z\le h-1)
	\ge(h-t)(1-1/h)\ge h-t-1,
	\]
	proving~\eqref{eq:poisson-constant-bounds}. Dividing these bounds by $h$ gives $c_h/h\to1$. Finally, the function $\lambda\me^{-\lambda}$ has derivative $(1-\lambda)\me^{-\lambda}$ and attains its maximum at $\lambda=1$, so $c_1=\me^{-1}$.
\end{proof}
}

\begin{proof}[Proof of Theorem~\ref{thm:mainPoisson}]
	Fix $k\ge3$, put $h:=\lfloor(k-1)/2\rfloor$, and fix $\lambda>0$. Let $r\to\infty$.  Let $q$ be the least prime with $q\ge r/\lambda$. The Prime Number Theorem implies that $q=(1+o(1)) r/\lambda$. Put
	\[
	N:=\lceil r^{5/2}\rceil.
	\]
	For the binomial random variable $Z\sim\operatorname{Bin}(r,1/q)$ and each fixed non-negative integer $j$, we have
	\[
	\mathbb P(Z=j)=\binom rj q^{-j}(1-q^{-1})^{r-j}
	\longrightarrow\me^{-\lambda}\frac{\lambda^j}{j!}.
	\]
	Also, by~\eqref{eq:product-bound},
	\[
	1\ge\prod_{i=2}^{r-1}(1-q^{-i})
	\ge1-\sum_{i=2}^{\infty}q^{-i}
	=1-\frac1{q(q-1)}\longrightarrow1,
	\]
	while
	\[
	rs(q)=\frac rq(1-q^{-1})(1+q^{-2})\longrightarrow\lambda.
	\]
	Two further applications of~\eqref{eq:product-bound} give
	\begin{align*}
		1\ge\frac{(N)_r}{N^r}
		&\ge1-\frac{r(r-1)}{2N}\longrightarrow1,\\
		1\ge(1-q^{-3})^{N-r}
		&\ge1-(N-r)q^{-3}\longrightarrow1.
	\end{align*}
	Substituting these limits into~\eqref{eq:poisson-exact} yields
	\[
	\liminf_{r\to\infty}r\pi(H_k^r)
	\ge\lambda\me^{-\lambda}\sum_{j=0}^{h-1}\frac{\lambda^j}{j!}.
	\]
	Since $\lambda>0$ was arbitrary, taking the supremum proves~\eqref{eq:poisson-fixed-bound}.
\end{proof}

\section{Lower bounds from the circular construction}\label{sec:3}

The main aim of this section is to prove Theorem~\ref{thm:main}, where we estimate the lower bound on $\pi(H_k^r)$ coming from the circular construction of Sidorenko~\cite{Sidorenko24}. 

We first recall this construction.
Let 
\[
\mathcal C:=\{z\in\mathbb C: |z|=1\}
\]
be the unit circle in the complex plane. Let $G_{r,k-1}$ be the (infinite) $r$-graph with vertex set $\mathcal C$, where an $r$-subset $A\subseteq \mathcal C$ is an edge if and only if 
\[\arg\Big(\prod_{z\in A}z \Big)\in\left[0, \Delta_{k-1}(A)\right],
\]
where $\Delta_{\ell}(A)$ denotes the shortest length of a closed arc in $\mathcal C$ that contains at least $\ell$ elements of $A$. Thus an $r$-set $A$ is an edge of $G_{r,k-1}$ if the sum of the angles of its elements modulo $2\pi$ is at most $\Delta_{k-1}(A)$.
Sidorenko~\cite[Theorem 3.1]{Sidorenko24} proved that this $r$-graph is $H_k^r$-free. 

Let $z_1,\dots,z_r$ be $r$ independent random points uniformly distributed on the circle~$\mathcal C$ and let $A:=\{z_1,\dots,z_r\}$; these points are almost surely distinct. When we apply a uniformly random rotation $\phi$ of the circle $\mathcal C$ to any fixed outcome of distinct points $z_1,\dots,z_r$, the rotated $r$-set $\phi(A)$ is an edge of $G_{r,k-1}$ with probability 
\[
\xi_{r,k-1}=\xi_{r,k-1}(z_1,\dots,z_r):=\frac{\Delta_{k-1}(A)}{2\pi}.
\]
This holds since $\arg(\prod_{j=1}^r \phi(z_j))$ is uniformly distributed in $[0,2\pi)$ while $\Delta_{k-1}(\phi(A))=\Delta_{k-1}(A)$ for any rotation~$\phi$. 

By considering the subgraph of $G_{r,k-1}$ induced by $n\to\infty$ independent uniformly distributed points on the circle and using the Fubini--Tonelli theorem, we obtain Inequality~\eqref{eq:GeneralLower} from the Introduction (which states that $\pi(H_k^r)$ is at least the expectation of the random variable $\xi_{r,k-1}$).
Thus our task is to estimate $\mathbb E(\xi_{r,k-1})$ for independent uniform $z_1,\dots,z_r\in\mathcal C$.

Let $z_1'=z_1$, and relabel the remaining points as $z_2',\dots,z_r'$ in counterclockwise order starting from $z_1$. When working with the variables $z_i'$ and various cyclic sums, we identify the index set $[r]=\{1,\dots,r\}$ with $\Z{r}$ and do index addition modulo $r$. 
Let $L_i$ be the length of the counterclockwise arc from $z_i'$ to $z_{i+1}'$ for $i\in [r]$. Then
\[\Delta_{k-1}(A)=\min_{ i\in [r]} \sum_{j=0}^{k-3}L_{i+j}.\]
Clearly, $L_1+\cdots +L_r=2\pi$.

The distribution of $(L_1,\dots,L_r)$ can be described using the following standard \emph{normalised exponential representation} (see Pyke~\cite[Section~4.1]{Pyke65} and compare with the related order-statistic representations of R\'enyi~\cite{Renyi53}).
Let $X_1,\ldots,X_r$ be independent exponential random variables with mean $1$, that is, each has density $\me^{-x}$ for $x\ge0$; extend their indices cyclically modulo $r$.  
Let 
\[
S:=\sum_{j=1}^rX_j.
\]
For $1\le\ell\le r-1$ and $i\in [r]$, we define
\begin{equation}\label{eq:WM}
	D_{i,\ell}:=\sum_{j=0}^{\ell-1}(1-X_{i+j})
	\quad \text{and} \quad
	M_{\ell}:=\max_{i\in [r]}D_{i,\ell}.
\end{equation}
The following result is standard; the spacing density claim used in its proof appears, for example, in Haj{\'o}s and R{\'e}nyi~\cite[Eq.~(6)]{HajosRenyi54}. For completeness, we give a proof.

\begin{lemma}\label{lem:renyi1953}
The vector $\big(\frac{L_1}{2\pi },\ldots, \frac{L_r}{2\pi }\big)$ has the same distribution as
$\big(\frac{X_1}{S},\ldots,\frac{X_r}{S}\big),$
and the latter normalised vector is independent of $S$.  Consequently, for $1\le \ell\le r-1$, we have
\begin{equation}\label{eq:EXi}
        \mathbb{E}(\xi_{r,\ell+1})=\frac{\ell-\mathbb{E}M_{\ell}}{r}.
\end{equation}
\end{lemma}

\begin{proof}
Conditional on $z_1$, the counterclockwise angles of the other $r-1$ points relative to $z_1$, divided by $2\pi$, are independent uniform variables on $[0,1)$. Their order statistics have joint density $(r-1)!$ on $0<u_1<\cdots<u_{r-1}<1$. The change of variables from these order statistics to the first $r-1$ consecutive spacings has Jacobian determinant 1. Thus $(L_1/(2\pi),\ldots,L_{r-1}/(2\pi))$ has constant density $(r-1)!$ on $\mathbb S_{r-1}\subseteq \mathbb R^{r-1}$, the simplex of positive coordinates summing to less than 1.

For $j\in[r]$, write $V_j:=X_j/S$. Thus $V_1+\dots+V_r=1$. At $(v_1,\ldots,v_{r-1},s)\in \mathbb S_{r-1}\times (0,\infty)$, the change of variables from $(V_1,\dots,V_{r-1},S)$ to $(X_1,\dots,X_r)$ has Jacobian determinant $s^{r-1}$, so the joint density of $(V_1,\dots,V_{r-1},S)$ is
\[
 \left(\prod_{i=1}^{r-1} \me^{-sv_i}\right) \me^{-s(1-v_1-\ldots-v_{r-1})}\, s^{r-1}
 =\frac{s^{r-1}\me^{-s}}{(r-1)!}\,(r-1)!
\]
(and the density is 0 elsewhere). This factorisation proves both the asserted distribution and the independence from $S$. 

Finally, define $f(v_1,\dots,v_r):=\min_{i\in[r]}\sum_{j=0}^{\ell-1}v_{i+j}$. Then 
\[
M_\ell=\ell-Sf(X_1/S,\ldots,X_r/S)
\]
 and $f(X_1/S,\ldots,X_r/S)$ has the distribution of $\xi_{r,\ell+1}$. Independence and $\mathbb E S=r$ give the claimed expectation identity~\eqref{eq:EXi}.
\end{proof}

Combining \eqref{eq:GeneralLower} and~\eqref{eq:EXi}, we obtain
\begin{equation}\label{eq:main-reduction}
 \pi(H_k^r)\ge\frac{k-2-\mathbb EM_{k-2}}r.
\end{equation}

Thus our aim is to upper bound the expectation of $M_{k-2}$. Informally speaking, we need to control the positive deviations of the centred sums $D_{i,k-2}$ simultaneously for all $i\in[r]$. Since $D_{i,\ell}$ is a sum of independent random variables, 
an obvious approach is to apply a concentration result (e.g.\ the Chernoff Bound) and then take the Union Bound over indices $i$. This indeed gives a reasonable upper bound on $\mathbb EM_{k-2}$ that is, in particular, strong enough to prove the lower bound of Corollary~\ref{cor:applications}\ref{it:Range1}, namely that 
 \begin{equation}\label{eq:Range1}
 	 \pi(H_k^r)\ge (1+o(1))k/r,\quad\text{when $k/\log r\to\infty$}.
 \end{equation} 
Compared with~\eqref{eq:Range1}, the essence of Theorem~\ref{thm:main} is to obtain good explicit lower-order terms. So, to illustrate some concepts and ideas used later, we first present a simple self-contained proof of a weaker result that implies~\eqref{eq:Range1}.
It is obtained by writing out the standard derivation of the Chernoff Bound via moment-generating functions,
with the Union Bound replaced by a max-versus-sum inequality (inside~\eqref{eq:UB}).

\begin{proposition}\label{pr:EM} For all integers $r\ge2$ and $3\le k\le r+1$, we have
\begin{equation}\label{eq:direct-circular-bound}
 \mathbb EM_{k-2}\le\sqrt{2(k-2)\log r}.
\end{equation}

\end{proposition}

\begin{proof}
Put $\ell:=k-2$. For every $t>0$,
\begin{equation}\label{eq:E(1-X)}
 \mathbb E\me^{t(1-X_1)}=\int_0^\infty \me^{t(1-x)}\,\me^{-x}\,dx=\frac{\me^t}{1+t}\le \me^{t^2/2},
 \end{equation}
 where the last inequality follows by
 \[ 
 t-\log(1+t)=\int_0^t\frac{u}{1+u}\,du\le \int_0^t {u}\,du=\frac{t^2}{2}.
\]
Each $D_{i,\ell}$ is a sum of independent random variables, so
\begin{equation}\label{eq:direct-window-mgf}
 \mathbb E\me^{tD_{i,\ell}} =\mathbb E\left(\prod_{j=0}^{\ell-1} \me^{t(1-X_{i+j})}\right)= \prod_{j=0}^{\ell-1}\mathbb E \me^{t(1-X_{i+j})}
 =\left(\frac{\me^t}{1+t}\right)^\ell
 \le\me^{\ell t^2/2}.
\end{equation}
Since $M_\ell=\max_{i\in[r]}D_{i,\ell}$, Jensen's inequality and~\eqref{eq:direct-window-mgf} give
\begin{equation}\label{eq:UB}
 \me^{t\mathbb EM_\ell}
 \le\mathbb E\me^{tM_\ell}
 = \mathbb E\left(\max_{i\in[r]}\me^{tD_{i,\ell}}\right)
 \le\mathbb E\left(\sum_{i=1}^r\me^{tD_{i,\ell}}\right)
 \le r\me^{\ell t^2/2}.
\end{equation}
Taking logarithms and choosing $t:=\sqrt{2\log r/\ell}$ gives
\[
 \mathbb EM_\ell\le\frac{\log r}{t}+\frac{\ell t}{2}
 =\sqrt{2\ell\log r},
\]
 which gives the claimed upper bound on the expectation of $M_\ell$.
\end{proof}

When $k=\Theta(r)$,  the  lower bound on $\pi(H_k^r)$ coming from Proposition~\ref{pr:EM}  has an error term of order $\sqrt{\log r/r}$.
In particular, it loses a factor $\sqrt{\log r}$ compared with
Theorem~\ref{thm:main} when $k/r$ stays bounded away from $0$ and~$1$.
The logarithm comes from bounding the maximum over all $r$ cyclic intervals of length $\ell$ by
the sum of $r$ exponential moments, without using their overlaps. The main idea for removing this factor is to use a maximal-type
inequality (namely, Lemma~\ref{lem:partial-max}) that controls at once all nested partial sums (of different lengths but with the same starting index). 
In order to control the cyclic intervals of length $\ell$, we consider overlapping blocks of intervals, split intervals in a block into a suffix and a prefix along a common cut and apply the maximal inequality on each side of the cut.
The proof is quite general and the only property of the random variable $X_1$ (and thus of each of its independent copies $X_i$) that we use is that for both $U:=X_1-1$ and $U:=1-X_1$, we have $\mathbb EU=0$ and
\begin{equation}\label{eq:one-sided-mgf}
 \log\mathbb E\me^{tU}\le\frac{t^2}{2(1-t)},
 \quad \text{for $0\le t<1$}.
\end{equation}
For the second choice this follows from~\eqref{eq:E(1-X)}, while
for the first it can be proved by e.g.\ comparing the corresponding power series term by term:
\[
 \log\mathbb E\me^{t(X_1-1)}
 =-t-\log(1-t)
 =\sum_{q=2}^{\infty}\frac{t^q}{q}
 \le \frac{t^2}{2} \sum_{q=0}^{\infty}{t^q}
 =\frac{t^2}{2(1-t)}.
\]

The next lemma follows directly from Doob's $L^2$ maximal inequality
(see e.g.~\cite[Theorem~4.4.4]{Durrett19pte}), applied to the nonnegative
submartingale $\me^{\theta T_j}$, and Cauchy--Schwarz.
We include a proof using the standard
first-crossing argument; compare with the proof of
\cite[Theorem~4.4.2]{Durrett19pte}.

\begin{lemma}\label{lem:partial-max}
Let $n\ge0$ be an integer and let $U_1,\ldots,U_n$ be independent copies of an integrable random
variable $U$ with $\mathbb EU=0$ satisfying~\eqref{eq:one-sided-mgf}.
Put $T_0:=0$ and $T_j:=\sum_{i=1}^jU_i$ for $j\in[n]$.
Then, for every $0<\theta<1/2$,
\begin{equation}\label{eq:partial-max-mgf}
 \mathbb E\exp\left(\theta\max_{0\le j\le n}T_j\right)
 \le2\exp\left(\frac{n\theta^2}{1-2\theta}\right).
\end{equation}
\end{lemma}

\begin{proof}
Put $b:=\mathbb E\me^{2\theta U}$ and
$M:=\max_{0\le j\le n}\me^{\theta T_j}$.
Jensen's inequality gives $b\ge\me^{2\theta\mathbb EU}=1$, and~\eqref{eq:one-sided-mgf}
ensures that $b<\infty$. For $v>0$, let $E_j$ be the event that
$j$ is the first index for which $\me^{\theta T_j}\ge v$.
The events $E_0,\ldots,E_n$ are disjoint and have union $\{M\ge v\}$.
By independence of the remaining increments,
\[
 \mathbb E\bigl(\me^{2\theta T_n}\mathbf1_{E_j}\bigr)
 =b^{n-j}\mathbb E\bigl(\me^{2\theta T_j}\mathbf1_{E_j}\bigr)\ge b^{n-j}v^2\mathbb P(E_j)
 \ge v^2\mathbb P(E_j).
\]
Summing over $j$ yields
$$
 v^2\mathbb P(M\ge v)
 \le\sum_{j=0}^n\mathbb E\bigl(\me^{2\theta T_n}\mathbf1_{E_j}\bigr)
 =\mathbb E\bigl(\me^{2\theta T_n}\mathbf1_{\{M\ge v\}}\bigr)
 \le\mathbb E\me^{2\theta T_n}=b^n.
$$ 
Hence
\[
 \mathbb EM
 \le\int_0^{\infty}\min\left\{1,\frac{b^n}{v^2}\right\}\,dv 
 =\int_0^{b^{n/2}}1\,dv
       +\int_{b^{n/2}}^{\infty}\frac{b^n}{v^2}\,dv
 =2b^{n/2}
 \le2\exp\left(\frac{n\theta^2}{1-2\theta}\right),
\]
where the last inequality is~\eqref{eq:one-sided-mgf} with $t=2\theta$.
\end{proof}

We now use Lemma~\ref{lem:partial-max} to control overlapping cyclic
intervals. Recall that empty sums are understood to be zero.

\begin{lemma}\label{lem:scan}
Let $r\ge2$ be an integer and let $U_1,\ldots,U_r$ be independent copies of an integrable random
variable $U$ with $\mathbb EU=0$ satisfying~\eqref{eq:one-sided-mgf}.
For an integer $1\le\ell\le r/2$, put
\[
 D'_{i,\ell}:=\sum_{j=0}^{\ell-1}U_{i+j},
 \quad\text{and}\quad 
 M'_\ell:=\max_{i\in[r]}D'_{i,\ell}.
\]
Then
\begin{align}
 \mathbb EM'_\ell
 \le2\sqrt{(2\ell-1)\log(4r/\ell)}+2\log(4r/\ell).
 \label{eq:scan-bound}
\end{align}
\end{lemma}

\begin{proof}
Fix $0<\theta<1/2$.
For $a\in[r]$, let
$V_a:=\max_{0\le j<\ell}D'_{a+j,\ell}$.
Any index $s$ attaining the maximum value of $D'_{s,\ell}$ appears (as $s=a+j$) in the definition of $\ell$ variables $V_a$.
 Therefore it holds deterministically that
\[
 \ell\me^{\theta M'_\ell}\le\sum_{a=1}^r\me^{\theta V_a}.
\]
The random variables $V_a$ have the same distribution by cyclic symmetry. Thus
\begin{equation}\label{eq:cyclic-averaging}
 \mathbb E\me^{\theta M'_\ell}
 \le\frac r\ell\,\mathbb E\me^{\theta V_1}.
\end{equation}

For $1\le i\le\ell$, by splitting the cyclic interval $\{i,\dots,i+\ell-1\}$ into two sub-intervals, we have
\[
 D'_{i,\ell}=\sum_{h=i}^{\ell-1}U_h
                  +\sum_{h=\ell}^{i+\ell-1}U_h.
\]
It follows that $V_1\le A+B$, where
\[
 A:=\max_{0\le j\le\ell-1}\sum_{h=1}^jU_{\ell-h},
 \qquad
 B:=\max_{0\le j\le\ell}\sum_{h=0}^{j-1}U_{\ell+h}.
\]
These are the maxima of partial sums on disjoint lists of lengths
$\ell-1$ and $\ell$, respectively, so they are independent.
Lemma~\ref{lem:partial-max} gives
\[
 \mathbb E\me^{\theta V_1}
 \le\mathbb E\me^{\theta( A+B)}=\mathbb E\me^{\theta A}\,\mathbb E\me^{\theta B}
 \le4\exp\left(\frac{(2\ell-1)\theta^2}{1-2\theta}\right).
\]
Together with Jensen's inequality and~\eqref{eq:cyclic-averaging},
this implies
\begin{equation}\label{eq:cyclic-optimize}
 \mathbb EM'_\ell\le \frac1\theta\log(\mathbb E\me^{\theta M'_\ell})\le\frac{\log(4r/\ell)}\theta
                  +\frac{(2\ell-1)\theta}{1-2\theta}.
\end{equation}
Put $a:=\log(4r/\ell)$ and $v:=2\ell-1$.
The choice $\theta:=\sqrt a/(\sqrt v+2\sqrt a)$ lies in $(0,1/2)$
and makes the right side of~\eqref{eq:cyclic-optimize} equal to
$2\sqrt{va}+2a$.
This proves~\eqref{eq:scan-bound}.
\end{proof}

\begin{proof}[Proof of Theorem~\ref{thm:main}]
Put $\ell:=k-2$ and $s:=\min\{\ell,r-\ell\}$.

First, suppose that $\ell\le r/2$. We take $U_j:=1-X_j$ in Lemma~\ref{lem:scan}. In the notation of the lemma, we have $D'_{i,\ell}=D_{i,\ell}$ for each $i\in [r]$ and
thus $M'_\ell=M_\ell$.
So the lemma implies (by $2s-1\le2s$ and $\log(4r/s)\le2\log(\me r/s)$) that
\begin{equation}\label{eq:minimum-lower-estimate}
 \mathbb EM_\ell
 \le4\left(\sqrt{s\log\frac{\me r}{s}}
                         +\log\frac{\me r}{s}\right).
\end{equation}

Next, suppose that $\ell>r/2$. Then $s=r-\ell$ and we take $U_j:=X_j-1$. It holds that $D'_{i,m}=-D_{i,m}$ for all $i$ and $m$. By considering complementary intervals, we have
\[
 D_{i,\ell}=r-S-D_{i+\ell,s},\quad\text{for every $i\in [r]$},
\]
and thus $M_\ell=r-S+M'_s$.
Taking expectations, using $\mathbb ES=r$, and
applying Lemma~\ref{lem:scan} gives~\eqref{eq:minimum-lower-estimate}
again. 

In either case,~\eqref{eq:minimum-lower-estimate} holds and its combination with~\eqref{eq:main-reduction} proves
the theorem.
\end{proof}

Next, we prove Corollary~\ref{cor:applications}, which illustrates the bound of Theorem~\ref{thm:main} in some ranges of $k=k(r)$.

\begin{proof}[Proof of Corollary \ref{cor:applications}]
	Let $s:=\min\{k-2,r-k+2\}$.
	By~\eqref{eq:(k-2)/r} and Theorem \ref{thm:main},
	\begin{equation}\label{eq:corollary-from-main}
		0\le \frac{k-2}{r}-\pi(H_k^r)
		\le \frac{4}{r}
		\left(
		\sqrt{s\log\frac{\me r}{s}}
		+\log\frac{\me r}{s}
		\right).
	\end{equation}
	
	We first assume that $k/\log r\to\infty$, which implies $\log(\me r/(k-2))=o(k)$.
	If $k-2\le r/2$, then $s=k-2$ and the right-hand side of~\eqref{eq:corollary-from-main} is clearly $o(k/r)$.
	If $k-2>r/2$, then $1\le s<k$, so
	$$\sqrt{s\log\frac{\me r}{s}}+\log\frac{\me r}{s}
	\le \sqrt{s\log(\me r)}+\log(\me r)
	=o(k).$$
	Thus, in both cases, the right-hand side of
	\eqref{eq:corollary-from-main} is $o(k/r)$, proving
	Part~\ref{it:Range1} of the corollary. (Alternatively, this part follows from Proposition~\ref{pr:EM}.)
	
	Finally, suppose that $k/r\to\alpha\in(0,1)$. Recall that $\beta=\min\{\alpha,1-\alpha\}$. 
	Then $s/r\to\beta$.
	Consequently,
	$$
	\frac{4}{r} \left(\sqrt{s\log\frac{\me r}{s}}+\log\frac{\me r}{s}\right)
	=\frac{4}{\sqrt r}\sqrt{\frac{s}{r}\log\frac{\me r}{s}}+O\left(\frac1r\right)
	=\frac{4\sqrt{\beta\log(\me /\beta)}+o(1)}{\sqrt r},$$
	as desired.
\end{proof}

Finally, we prove Corollary~\ref{cor:half-optimal}, which estimates the ratio between the upper bound in~\eqref{eq:(k-2)/r} and the new lower bounds as $r\to\infty$ for arbitrary $k=k(r)$. 
We will be rather brief and omit routine details.

\begin{proof}[Proof of Corollary~\ref{cor:half-optimal}]
	Let $r$ be sufficiently large. By Theorem~\ref{thm:main}, 
	it is enough to consider, say,  $k(r)\le\sqrt r$, as the bound from the theorem is uniform for $k>\sqrt r$.

Let $G$ be the union of $b:=\lfloor(k-1)/2\rfloor$ independently and uniformly permuted copies of an $H_3^r$-free $r$-graph $G'$ on $n\to\infty$ vertices of edge density $p=(\me^{-1}-o(1))/r$. Such an $r$-graph $G'$ exists by Theorem~\ref{thm:mainPoisson}, after deleting edges if necessary. Every $(r+1)$-set spans at most $2b\le k-1$ edges in $G$ and the expected edge density of $G$ is
	\[
	1-(1-p)^b\ge bp-\binom b2p^2
	\ge\left(\frac1{2\me}-o(1)\right)\frac{k-2}{r},
	\]
	uniformly for $3\le k\le \sqrt{r}$. This implies the first assertion.
	
	For the second assertion, first note that $c_h\ge (1-o(1))h$ as $h\to\infty$ which can be seen by taking, for example, $\lambda:=h-h^{2/3}$. 
	\hide{The upper bound $c_h\le h$ follows from
		\[
		\lambda\mathbb P(Z\le h-1)
		=\mathbb E[Z\mathbf 1_{\{Z\le h\}}]\le h
		\quad\text{for }Z\sim\operatorname{Pois}(\lambda).
		\]
		
		For the lower bound, let $\lambda:=h-h^{2/3}$. Chebyshev's inequality gives $\mathbb P(Z\ge h)\le\lambda/h^{4/3}\le h^{-1/3}$, so
		\[
		c_h\ge(h-h^{2/3})(1-h^{-1/3})=(1-o(1))h.
		\]
	}
	Given $\varepsilon>0$, fix an integer $h$ sufficiently large and then a real $\eta>0$ sufficiently small so that $(c_h-\eta)/(2h)>1/2-\varepsilon/2$.
	For all sufficiently large $r$, let $G$ be the union of  $b:=\lfloor(k-1)/(2h)\rfloor$ independently and uniformly permuted copies of an $H_{2h+1}^r$-free $r$-graph of edge density $p\ge (c_h-\eta)/r$, which exists by Theorem~\ref{thm:mainPoisson}. Each $(r+1)$-set contains at most $2hb\le k-1$ edges in $G$ and the expected edge density of $G$ is at least
	\[
	\frac{b(c_h-\eta)}r-O\left(\frac{b^2}{r^2}\right)
	=\left(\frac{c_h-\eta}{2h}-o(1)\right)\frac{k}{r}\ge \left(\frac12-\varepsilon\right)\frac{k}{r},
	\]
	implying the second assertion.
\end{proof}

\section{Augmentations of \texorpdfstring{$t$}{t}-graphs with zero kernel}\label{sec:2}


In this section, we fix a non-empty $t$-graph $F$ of uniformity $t\ge 2$. Recall that $B_F$ denotes the edge--vertex incidence matrix of $F$ and, for $r\ge t$, the $r$-augmentation $F(r)$ is the $r$-graph obtained from $F$ by adding the same set of $r-t$ new vertices to every edge of~$F$. We will refer to the set of the new $r-t$ vertices as the \emph{stem} of $F(r)$. The vertices in $V(F)$ will be called \emph{base} vertices.

\hide{
It will be convenient to operate with the following property, which we show in Lemma~\ref{lem:app-rank-rigidity} to be equivalent to $B_F$ having zero kernel.

\begin{definition}\label{def:app-sum-rigid}
Call the  $t$-graph $F$  \emph{sum-rigid} if, whenever real numbers $d$ and $(x_v)_{v\in V(F)}$ satisfy
\[
 \sum_{v\in e}x_v=d
 \quad\text{for every }e\in E(F),
\]
then all the numbers $x_v$ for $v\in V(F)$ are equal.
\end{definition}

\begin{lemma}\label{lem:app-rank-rigidity}
The $t$-graph $F$ is sum-rigid if and only if the matrix $B_F$ has zero kernel over $\mathbb R$.
\end{lemma}

\begin{proof} Suppose first that $B_F$ has zero kernel and that all edge sums of a vector $\V{x}\in \mathbb R^{V(F)}$ are equal to $d$.  Since every edge of $F$ has $t$ vertices, it holds that
$B_F\mathbf 1=t\mathbf 1$.
Hence
\[
 B_F\left(\V{x}-\frac dt\mathbf 1\right)=\V{0}.
\]
The kernel of $B_F$ is trivial by our assumption, so $\V{x}=(d/t)\mathbf 1$ and $F$ is sum-rigid.

Conversely, suppose that $F$ is sum-rigid and $B_F\V{x}=0$ for some vector $\V{x}\in \mathbb R^{V(F)}$.  All edge sums of $\V{x}$ are equal to zero, so $\V{x}=c\mathbf 1$ for some $c\in\mathbb R$ by the rigidity assumption.  Then
\[
 0=B_F\V{x}=B_F (c\mathbf 1)=cB_F\mathbf 1=ct\mathbf 1,
\]
which gives $c=0$.  Thus the kernel of $B_F$ is trivial, as desired.
\end{proof}
}

Our proof of Theorem~\ref{end:thm:main} will take some prime $q=q(r)$ depending on $r$ (in order to define the first checksum in the finite field $\mathbb F_q$ with $q$ elements). We will need $q$ to satisfy certain properties.  The following lemma lists these properties (with the third one being crucial in the proof of Theorem~\ref{end:thm:main}) and shows that there are only finitely many exceptional primes.

\begin{lemma}\label{lem:app-good-characteristics}
Suppose that $B_F$ has zero kernel over $\mathbb R$.  Then there is an integer $Q_F$ such that, for every prime $q> Q_F$,
\begin{enumerate}[label=\textup{(\roman*)}]
\item\label{it:i} $t$ is non-zero in $\mathbb F_q$;
\item\label{it:ii} $B_F$ has zero kernel over $\mathbb F_q$;
\item\label{it:iii} if $\V{x}\in\mathbb F_q^{V(F)}$ and $B_F\V{x}=d\mathbf 1$ for some $d\in\mathbb F_q$, then $\V{x}=(d/t)\mathbf 1$.
\end{enumerate}
\end{lemma}

\begin{proof}
Let $h:=|V(F)|$.  Since $B_F$ has rank $h$ over $\mathbb R$, it has an $h\times h$ minor with non-zero determinant~$D$. Since the matrix $B_F$ is $\{0,1\}$-valued, $D$ is an integer. Let $Q_F$ be the largest prime divisor of $t|D|$. Let us verify that it satisfies the required properties. Take any prime $q> Q_F$. Then $t$ and $D$ are non-zero as elements of $\mathbb F_q$ and the above $h\times h$ minor remains nonsingular over $\mathbb F_q$, proving Items~\ref{it:i} and~\ref{it:ii}. The last item is a consequence of the first two and the identity $B_F\mathbf 1=t\mathbf 1$: if $B_F\V{x}=d\mathbf 1$ then
\[
 B_F\left(\V{x}-\frac dt\mathbf 1\right)
 =d\mathbf 1-\frac dt B_F\mathbf 1=\V{0}
\]
and the zero kernel property over $\mathbb F_q$ gives $\V{x}=(d/t)\mathbf 1$, as desired.
\end{proof}

We also need a few auxiliary combinatorial lemmas. The following lemma shows that the zero kernel condition guarantees a positive Tur\'an density for each augmentation.

\begin{lemma}\label{lem:app-augmentation-positive}
	If $B_F$ has zero kernel over $\mathbb R$, then,  for every $j\ge t$,  the $j$-graph $F(j)$ is not $j$-partite and, consequently,
	\begin{equation*}
		\pi(F(j))\ge\frac{j!}{j^j}>0.
	\end{equation*}
\end{lemma}

\begin{proof}
	First, observe that the original $t$-graph $F$ has no isolated vertices, since an isolated vertex would give a zero column of $B_F$.  Also, $F$ is not $t$-partite.  Indeed, if
	\(
	V(F)=V_1\cup\cdots\cup V_t
	\)
	were a $t$-partition with every edge containing one vertex from every part, then the vector which is $1$ on $V_1$, $-1$ on $V_2$, and $0$ elsewhere would be a non-zero element of the kernel of $B_F$.
	
	Suppose on the contrary that, for some $j\ge t$, the $j$-graph $F(j)$ is $j$-partite.  Its $j-t$ stem vertices occur together in every edge and must therefore lie in distinct vertex classes.  Every base vertex belongs to an edge, so the base vertices must lie inside the remaining $t$ classes, and every base edge contains exactly one vertex from each of them.  Thus $F$ is $t$-partite. This contradiction proves the first part of the lemma.
	
	For the second part, consider the complete balanced $j$-partite $j$-graph with $n\to\infty$ vertices. It is $F(j)$-free and its edge density is $j!/j^j+o(1)$, finishing the proof of the lemma.
\end{proof}

The standard averaging argument over vertex links can be applied to augmentations, giving the following monotonicity property.

\begin{lemma}\label{lem:app-augmentation-monotonicity}
	For every $j\ge t$,
	\begin{equation*}
		\pi(F(j))\ge\pi(F(j+1)).
	\end{equation*}
\end{lemma}

\begin{proof}
	Let $\mathcal H$ be an $F(j+1)$-free $(j+1)$-graph on $n$ vertices with $\mathrm{ex}(n,F(j+1))$ edges.  For each vertex $v$, its \emph{link}
	\[
	L_{\mathcal H}(v)
	:=\left\{A\in\tbinom{V(\mathcal H)\setminus\{v\}}j:A\cup\{v\}\in\mathcal H\right\}
	\]
	is $F(j)$-free: a copy of $F(j)$ in the link, with $v$ adjoined to every edge, would give a copy of $F(j+1)$ in $\mathcal H$.  Therefore
	\[
	(j+1)|\mathcal H|
	=\sum_{v\in V(\mathcal H)}|L_{\mathcal H}(v)|
	\le n\,\mathrm{ex}(n-1,F(j)).
	\]
	After division by $(j+1)\binom n{j+1}$, this becomes
	\[
	\frac{\mathrm{ex}(n,F(j+1))}{\binom{n}{j+1}}=\frac{|\mathcal H|}{\binom n{j+1}}
	\le
	\frac{\mathrm{ex}(n-1,F(j))}{\binom{n-1}j}.
	\]
	The lemma follows by taking $n\to\infty$. 
\end{proof}

The following well-known fact, proved by Katona, Nemetz and Simonovits~\cite{KatonaNemetzSimonovits64}, will also be used.

\begin{lemma}\label{end:lem:finite-density}
Let $J$ be a $j$-graph. Then the ratio ${\mathrm{ex}(n,J)}/{\binom nj}$ is nonincreasing for $n\ge j$. In particular, $\mathrm{ex}(n,J)/\binom{n}{j}\ge \pi(J)$ for every $n\ge j$.
\end{lemma}

\begin{proof}
For $n\ge j+1$, if $\mathcal H$ is an extremal $J$-free $j$-graph on $n$ vertices, then the subgraph induced by a uniformly random $(n-1)$-subset has expected size
\(
 (n-j)|\mathcal H|/n.
\)
Consequently,
\[
 \mathrm{ex}(n-1,J)\ge\frac{n-j}{n}\,\mathrm{ex}(n,J).
\]
Since $\binom{n-1}j=\frac{n-j}{n}\binom nj$, division by $\binom{n-1}j$ proves the monotonicity.  
\end{proof}

For a given $j$-graph $J$ and integers $\PP,n$, 
the following simple probabilistic argument gives $\PP$ $J$-free $j$-graphs  on the common vertex set $[n]$ whose union is ``large''. We find it convenient to assume that the constructed $j$-graphs are edge-disjoint and encode them by a single map $\varphi:\binom{[n]}j\longrightarrow
\Z{\PP}\cup\{\infty\}$. Recall that $\Z{\PP}=\{0,\dots,\PP-1\}$ denotes the set of residues modulo $\PP$. Here the $j$-graphs are the preimages of the elements of $\Z{\PP}$ under $\varphi$ while the set of uncovered edges is $\varphi^{-1}(\infty)$.

\begin{lemma}\label{lem:app-global-colouring}
Let $J$ be a fixed $j$-graph, let $0<\beta<\pi(J)$, and let $n,\PP$ be positive integers with $n\ge j$.  Then there is a random map
\[
 \varphi:\binom{[n]}j\longrightarrow
 \Z{\PP}\cup\{\infty\}
\]
with all of the following properties.
\begin{enumerate}[label=\textup{(\roman*)}]
\item\label{it:JFree} For every outcome and every $b\in\Z{\PP}$, the colour class $\varphi^{-1}(b)$ is $J$-free.
\item\label{it:Uncovered} For every fixed $B\in\binom{[n]}j$,
\begin{equation*}
 \mathbb P(\varphi(B)=\infty)\le \me^{-\beta \PP}.
\end{equation*}
\item\label{it:Invariance} The distribution of $\varphi$ is invariant under every permutation of $[n]$.
\end{enumerate}
\end{lemma}

\begin{proof}
Take a maximum $J$-free $j$-graph $\mathcal I$ on $[n]$.  By Lemma~\ref{end:lem:finite-density},
\[
 \frac{|\mathcal I|}{\binom nj}\ge\pi(J)>\beta.
\]
For each $b\in\{0,\ldots,\PP-1\}$, choose a bijection
\[
 \psi_b:[n]\longrightarrow[n]
\]
 uniformly at random, with the choices for different $b$ being independent.
For $B\in\binom{[n]}j$, let $\varphi(B)$ be the least $b\in\{0,\ldots,\PP-1\}$ for which $\psi_b(B)\in\mathcal I$ and view $b$ as an element of $\Z{\PP}$. We define $\varphi(B):=\infty$ if no such $b$ exists.

For fixed $B\in\binom{[n]}{j}$, the set $\psi_b(B)$ is a uniformly random $j$-subset of $[n]$.  Each colour therefore succeeds in colouring $B$ with probability $|\mathcal I|/\binom nj>\beta$.  Hence, by independence,
\[
 \mathbb P(\varphi(B)=\infty)
 \le(1-\beta)^\PP\le \me^{-\beta \PP}.
\]

For any finite colour $b\in\Z{\PP}$, the colour-$b$ class $\varphi^{-1}(b)$ is contained in
\[
 \psi_b^{-1}(\mathcal I)
 :=\left\{B\in\tbinom{[n]}j:\psi_b(B)\in\mathcal I\right\}.
\]
 This is just a relabelling of $\mathcal I$ so it is $J$-free.

Finally, precomposing all functions $\psi_b$ with a fixed permutation of $[n]$ does not change their joint distribution.  This establishes the permutation invariance of~$\varphi$.
\end{proof}

The following result provides the key step of the proof.

\begin{theorem}\label{thm:app-lifting}
Suppose that $B_F$ has zero kernel over $\mathbb R$ and let $Q_F$ be the integer returned by Lemma~\ref{lem:app-good-characteristics}. 
Let integers $r>\ell\ge t$ satisfy $Q(r,\ell)\ge Q_F$, where 
\begin{equation}\label{eq:app-Q}
	Q(r,\ell):=
	\left\lceil
	\left(4\binom{r}{\ell+1}\right)^{1/\ell}
	\right\rceil.
\end{equation}
If $q>Q(r,\ell)$ is a prime then
\begin{equation}\label{eq:app-lifting}
 \pi(F(r))
 \ge
 \frac{3\pi(F(\ell))}{32q\log(2q)}.
\end{equation}
\end{theorem}

\begin{proof}
Set
\[
 \beta:=\frac{\pi(F(\ell))}{2}>0
 \quad\text{and}\quad
 \PP:=\left\lceil\beta^{-1}\log(2q)\right\rceil.
\]
Since $q\ge2$ and $\beta\le1/2$, the expression inside the ceiling is larger than $1$.  Therefore
\begin{equation}\label{eq:app-P-bound}
 \PP\le2\beta^{-1}\log(2q)
 =\frac{4\log(2q)}{\pi(F(\ell))}.
\end{equation}

Take any $n\ge r$. Let
\[
 a:[n]\longrightarrow\mathbb F_q,
\]
 be a map chosen uniformly at random among all $q^n$ choices.
The \emph{first checksum} of $A\in \binom{[n]}{r}$ is
\[
\CSOne(A):=\sum_{v\in A} a(v)\in\mathbb F_q.
\]  

The second checksum $\CSTwo(A)$ will be defined only for some $r$-sets $A$ and we need to introduce some further random objects in addition to $a$ in order to define it. For every $j\in\{t,\ldots,\ell\}$, Lemma~\ref{lem:app-augmentation-monotonicity} gives
\(
\pi(F(j))\ge\pi(F(\ell))>\beta.
\)
Apply Lemma~\ref{lem:app-global-colouring} to $J=F(j)$, with the same parameters $\beta,n,\PP$, and take the resulting random maps
\[
\varphi_j:\binom{[n]}j\longrightarrow
\Z{\PP}\cup\{\infty\},\quad \text{for  $j\in \{t,\ldots,\ell\}$},
\]
with the choices for different values of $j$ being mutually independent, and also independent of~$a$. For $A\subseteq[n]$ and $z\in\mathbb F_q$, define
\[
 A_z:=\{v\in A:a(v)=z\}
.
\]
 Thus $A_z$ is the intersection of $A$ with the fibre $a^{-1}(z)$.
 
We would like to define the  \emph{second checksum} of $A\in \binom{[n]}r$ as
 \[
 \CSTwo(A):=\sum_{\substack{z\in\mathbb F_q\\ \n{z}{A}\ge t}}\varphi_{\n{z}{A}}(A_z).
 \] 
 In order for this expression to make sense, we need that
 \[
 \text{$\n{z}{A}\le \ell$ and if $\n{z}{A}\ge t$ then $\varphi_{\n{z}{A}}(A_z)\not=\infty$,\quad for every $z\in\mathbb F_q$.}
 \] In this case, we say that $\CSTwo(A)$ \emph{is defined} and view it as an element of $\Z{\PP}$.

For each $\kappa_1\in\mathbb F_q$ and $\kappa_2\in \Z{\PP}$, define an $r$-graph $\F{\kappa_1}{\kappa_2}$ on $[n]$ where an $r$-set $A\subseteq [n]$ is an edge of $\F{\kappa_1}{\kappa_2}$ if and only if
\[
\text{$\CSOne(A)=\kappa_1$\quad and\quad $\CSTwo(A)=\kappa_2$,}
	\]
	that is, the first checksum of $A$ is $\kappa_1$ while the second checksum is defined and is equal to $\kappa_2$.

We first show that every realisation of $\F{\kappa_1}{\kappa_2}$ is $F(r)$-free.  Suppose to the contrary that it contains a copy of $F(r)$.  Let $D\subseteq [n]$ be the common stem of this copy, of size $r-t$, and identify its base vertices with $V(F)$.

For every $e\in E(F)$, the augmented edge $D\cup e$ has the first checksum $\kappa_1$, so
\[
 \sum_{u\in D}a(u)+\sum_{v\in e}a(v)=\kappa_1.
\]
Thus all base-edge sums of the vector $(a(v))_{v\in V(F)}$ are equal to $d:=\kappa_1-\sum_{u\in D}a(u)$.
Since $q>Q_F$, we have by Lemma~\ref{lem:app-good-characteristics}\ref{it:iii} that 
 $a(v)=d/t$ for every $v\in V(F)$.
Let
\[
 D':=\{u\in D:a(u)=d/t\}
 \quad\text{and}\quad t':=|D'|+t.
\]
 Note that, trivially, $t'\ge t$.
For every augmented edge $A=D\cup e$, the set of its vertices $v$ with $a(v)=d/t$ is exactly the $t'$-set $D'\cup e$, that is, $A_{d/t}=D'\cup e$. Since the second checksum of $A$ is defined, we have that $t'\le \ell$.   As $e$ ranges over $E(F)$, the sets $D'\cup e$ form a copy of $F(t')$ in the fibre $a^{-1}(d/t)\subseteq [n]$. 

All augmented edges $A$ of the copy of $F(r)$ have the same  occupancy profile $(\n{z}{A})_{z\in\mathbb F_q}$.  In every fibre other than $a^{-1}(d/t)$, they contain exactly the same subset of the common stem $D$. Their second checksums are defined and are all equal to each other. By cancelling the contributions from the fibres other than $a^{-1}(d/t)$, we derive that there is $b\in \Z{\PP}$ such that $\varphi_{t'}(D'\cup e)=b$ for every $e\in E(F)$.
We have therefore obtained a copy of $F(t')$ in $\varphi_{t'}^{-1}(b)$ with $b\not=\infty$, contradicting Lemma~\ref{lem:app-global-colouring}.  Thus $\F{\kappa_1}{\kappa_2}$ is $F(r)$-free.

Let us lower bound the expected edge density of $\F{\kappa_1}{\kappa_2}$ over uniform $(\kappa_1,\kappa_2)\in \mathbb F_q\times\Z{\PP}$, independent of the previous choices. Fix an $r$-set $A$.

If some $a$-fibre contains at least $\ell+1$ vertices of $A$, then the function $a$ is constant on some $(\ell+1)$-subset of $A$. For any fixed $(\ell+1)$-subset, this has probability $q^{-\ell}$.  Hence, by the Union Bound,
\begin{equation*}
 \mathbb P\bigl(\n{z}{A}>\ell\text{ for some }z\in\mathbb F_q\bigr)
 \le\binom r{\ell+1}q^{-\ell}
 \le\binom r{\ell+1}Q(r,\ell)^{-\ell}
 \le\frac14.
\end{equation*}

Let us condition on a restriction of $a$ to $A$ such that
each $\n{z}{A}$ is at most $\ell$. For each $z\in\mathbb F_q$ with $\n{z}{A}\ge t$, the conditional probability of the bad event that $\varphi_{\n{z}{A}}(A_z)=\infty$ is at most $\me^{-\beta \PP}$ by Lemma~\ref{lem:app-global-colouring}. 
By the Union Bound, the probability over the random functions $(\varphi_j)_{j=t}^{\ell}$ that there is $z\in\mathbb F_q$ with $\n{z}{A}\ge t$ and $\varphi_{\n{z}{A}}(A_z)=\infty$
is at most $q\me^{-\beta \PP}\le 1/2$. 

Suppose that no bad event as above occurs. Then the second checksum of the $r$-set $A$ is defined and assumes one of the $\PP$ possible values in $\Z{\PP}$. Since the first checksum $\CSOne(A)\in\mathbb F_q$ is always defined, there is a choice of $(\kappa_1,\kappa_2)$, namely $(\CSOne(A),\CSTwo(A))$, such that $A\in\F{\kappa_1}{\kappa_2}$.

Putting everything together, the probability over the random choices of $(a,(\varphi_j),\kappa_1,\kappa_2)$ that a given $r$-set $A\subseteq [n]$ is an edge of $\F{\kappa_1}{\kappa_2}$ is at least
\begin{equation*}
	\left(1-\frac14\right)\cdot\left(1-\frac1{2}\right)\cdot \frac1q\cdot \frac1{\PP}
 =\frac3{8q\PP}.
\end{equation*}
 By the linearity of expectation, there is at least one outcome when $|\F{\kappa_1}{\kappa_2}|\ge \frac3{8\PP q} \binom{n}{r}$. 
 Since this is true for every $n\ge r$ and every such realisation is $F(r)$-free, we conclude that
\[
 \pi(F(r))\ge\frac3{8q\PP}.
\]
Using \eqref{eq:app-P-bound}, we obtain \eqref{eq:app-lifting}.
\end{proof}

Let us make some brief remarks about the above proof.
The definition  of the first checksum $\CSOne$ and its usage (forcing all base vertices to have the same colour under the map $a$) are fairly standard, e.g.\ it appeared in the papers~\cite{KimRoush83,FranklRodl85gc,Sidorenko97} that dealt with lower bounds on $\pi(H_{r+1}^r)$. (These papers had a deterministic balanced colouring $a$ and our choice of making the function $a$ random was just a matter of convenience.) Here the zero-kernel condition ensures that equality of the first checksums forces all base vertices to have the same colour. 
The second checksum is more interesting and subtle. Informally, the main idea behind it is to extend the $a$-monochromatic base copy of $F$ to the maximal monochromatic augmentation (with the stem being $D'$ in the above proof) under the vertex colouring $a$ and to explore the consequences of all edges of the original $r$-augmentation $F(r)$ having the same $\CSTwo$-value. A related use of sums of local colours and cancellation appears in a paper of Sales~\cite[Section~3.3 and Proposition~5.6]{Sales24} on the Ramsey numbers of daisies.

The following intermediate result, which is a direct consequence of Theorem~\ref{thm:app-lifting}, will be useful in proving the  lower bound of Theorem~\ref{end:thm:main}.

\begin{lemma}\label{lem:app-one-step}
Suppose that $B_F$ has zero kernel over $\mathbb R$. If $r$ is sufficiently large, depending only on $F$, then it holds with $\ell:=\lceil8\log r\rceil$ that
\begin{equation}\label{eq:app-one-step}
 \pi(F(r))\ge\frac{3\pi(F(\ell))}{64r}.
\end{equation}
\end{lemma}

\begin{proof}
Let $Q=Q(r,\ell)$ be defined by \eqref{eq:app-Q}.  The standard binomial estimate gives
\begin{align*}
 \left(4\binom r{\ell+1}\right)^{1/\ell}
 &\le
 \left(4\left(\frac{\me r}{\ell+1}\right)^{\ell+1}\right)^{1/\ell}\\
 &=\frac{\me r}{\ell+1}
 \exp\left(
  \frac{\log4+\log(\me r/(\ell+1))}{\ell}
 \right).
\end{align*}
For all sufficiently large $r$,
\[
 \log4+\log\left(\frac{\me r}{\ell+1}\right)
 \le2\log r\le\frac \ell4.
\]
Consequently,
\[
 \left(4\binom r{\ell+1}\right)^{1/\ell}
 \le \frac {\me^{5/4}r}{\ell+1}<\frac {4r}\ell.
\]
Since $r/\ell\to\infty$, the ceiling contributes at most another $r/\ell$ for all sufficiently large $r$, and hence
$ Q\le {5r}/{\ell}$.
Also $Q\to\infty$, so it eventually exceeds the constant $Q_F$ returned by Lemma~\ref{lem:app-good-characteristics}. Take a prime $q$ with $Q<q<2Q$; it exists by Bertrand's postulate.  Then $q\le{10r}/{\ell}$.
For large $r$, we have $2q\le r$ and therefore
\[
 \log(2q)\le\log r\le\frac \ell8.
\]
It follows that
\[
 q\log(2q)
 \le\frac{10r}{\ell}\cdot\frac \ell8
 \le2r.
\]
Theorem~\ref{thm:app-lifting} now yields this lemma.
\end{proof}

\renewcommand{\rho}{R}

\begin{proof}[Proof of Theorem~\ref{end:thm:main}]
Fix $\rho=\rho(F)$ large enough so that, in particular, Lemma~\ref{lem:app-one-step} applies for every $r\ge \rho$. Next, let $r$ be sufficiently large. Define
\begin{equation*}
 r_1:=r
 \quad\text{and}\quad
 r_{i+1}:=\lceil8\log r_i\rceil\quad\text{for $i=1,2,3$.}
\end{equation*}
We can assume that, for each $i\in [4]$, it holds that $r_i\ge \rho$.
Applying \eqref{eq:app-one-step}
three times, we obtain
\begin{equation}\label{eq:app-three-steps}
 \pi(F(r))
 \ge \frac{3\pi(F(r_2))}{64\,r_1}
 \ge \left(\frac3{64}\right)^{2}\frac{\pi(F(r_3))}{r_1r_2}
  \ge \left(\frac3{64}\right)^{3}\frac{\pi(F(r_4))}{r_1r_2r_3}.
\end{equation}
By Lemma~\ref{lem:app-augmentation-positive} and the inequality $j!\ge (j/\me)^{j}$, which holds for every integer
$j\ge1$, we have
\begin{equation}\label{eq:app-base-case}
 \pi(F(r_4)) \ge \frac{r_4!}{r_4^{\,r_4}}\ \ge\ \me^{-r_4}.
\end{equation}
 Also, we can assume that $r_{i+1}\le 9\log r_i$ for each $i\in [3]$. Thus $\me^{r_4}\le
 r_3^{9}=(\log r)^{o(1)}$.  Substituting this into
\eqref{eq:app-base-case} and then into \eqref{eq:app-three-steps}, we conclude that
\[ \pi(F(r))
 \ge \left(\frac3{64}\right)^{3}\frac1{r_1\,r_2\,r_3^{10}}
 \ge \frac1{r(\log r)^{1+o(1)}},
\]
finishing the proof.\end{proof}

Finally,  Theorem~\ref{end:thm:augmentations}
follows from Theorem \ref{end:thm:main} and the backward implication in the following easy result.

\begin{lemma}\label{lem:app-graph-rank}
Let $F$ be a connected graph with at least one edge.  Then $B_F$ has zero kernel if and only if $F$ is non-bipartite.
\end{lemma}

\begin{proof}
Suppose first that $F$ is bipartite with parts $X$ and $Y$.  The non-zero vector $\V{x}$ defined by
\[
 x_v:=\left\{\begin{array}{rl}
 	 1,&\text{if $v\in X$},\\
-1,&\text{if $v\in Y$,}
\end{array}\right.
\]
lies in the kernel of $B_F$, so $B_F$ does not have zero kernel.

Conversely, suppose that $F$ is non-bipartite and $B_F\V{x}=0$.  For every edge $uv$,
\begin{equation}\label{eq:app-graph-edge}
 x_u+x_v=0.
\end{equation}
Thus the values alternate in sign along any path.  Since $F$ contains an odd cycle, following \eqref{eq:app-graph-edge} around that cycle gives $x_v=-x_v$ at a cycle vertex, and hence $x_v=0$.  Connectedness then propagates zero along paths to every vertex.  Thus $B_F$ has zero kernel.
\end{proof}

\section{Concluding remarks}
\label{sec:concluding}

It is interesting that all the main results of this paper (including the analysis of the circular construction of Sidorenko) use probabilistic methods.

Of course, an obvious intriguing open problem is to determine the asymptotic value of $\pi(H_3^r)$ as $r\to\infty$.

In terms of $r$-augmentations of a fixed $3$-chromatic graph $F$, \eqref{eq:(k-2)/r} and Theorem~\ref{thm:mainPoisson} imply that $\pi(F(r))=\Theta(1/r)$ if $F$ has a triangle. In the next simplest case when $F$ is the $5$-cycle $C_5$, we have a multiplicative gap of $(\log r)^{1+o(1)}$ between the bounds and closing it is an interesting open problem.

\begin{problem} Determine $\pi(C_5(r))$ within a constant factor as $r\to \infty$.\end{problem}

A natural approach to improving the lower bound in Theorem~\ref{thm:mainPoisson} is to replace the cap set $\mathcal C_q$ by a set containing no $k$ collinear points and take $h=1$ in Lemma~\ref{lem:poisson-free} (or to use a combination of relaxing the condition on the set $\mathcal C_q$ and decreasing $h$ appropriately). We do not know what general improvement can be achieved this way, so we briefly discuss just the case $k=4$ (for which the multiplicative gap between the current asymptotic bounds is the largest). Here, the best lower bound coming from the results of this paper is 
 \begin{equation}\label{eq:H4r}
 \pi(H_4^r)\ge \frac{\me^{-1}-o(1)}r,
 \end{equation} the same as the bound for $\pi(H_3^r)$. Define $m(q)$ to be the maximum size of a set $\mathcal X\subseteq \operatorname{PG}(3,q)$ that contains no 4 collinear points. If we can improve the bound $m(q)\ge q^2+1$ coming from Lemma~\ref{lem:cap} by a factor $1+\Omega(1)$ for a sequence of prime powers $q$ whose consecutive terms have ratio $1+o(1)$ then this will improve~\eqref{eq:H4r}. Such an improvement is consistent with our current knowledge:  Bramwell and Wilson~\cite{BramwellWilson73} proved that $m(q)\le 2q^2$ for $q\ge4$, while the best known general upper bound by Hill~\cite{Hill78} improves it only in lower order terms. The best  lower bound valid for every large prime $q$ that the authors are aware of is $m(q)\ge (q+1)^2$ coming from a smooth cubic surface $\mathcal X$ of Frobenius class $C_{12}$, see Banwait, Fit\'e and Loughran~\cite[Eq.~(2.1) and Table~7.1]{BanwaitFiteLoughran19} and Loughran and Trepalin~\cite[Theorem~1.1]{LoughranTrepalin20}. Such a surface has $(q+1)^2$ rational points and contains no line defined over $\mathbb F_q$. Restricting its cubic equation to any line defined over $\mathbb F_q$ therefore gives at most three zeros.
 \hide{Some details:
 	Hill~\cite{Hill78} improved it to $2q^2-\alpha_q$ where $\alpha_q$ comes from an explicit (but messy) inequality and is linear in $q$, specifically he showed that
 	\begin{equation}\label{eq:H4-Hill-bound}
 		m(q)\le2q^2-\frac{\sqrt{17}-3}{2}\,q+O(1)
 		\quad(q\to\infty);
 	\end{equation}

 	In the other direction, results on cubic surfaces imply
 	\begin{equation}\label{eq:H4-cubic-cap}
 		m(q)\ge(q+1)^2
 		\quad\text{for every prime power }q.
 	\end{equation}
 	To see this, Theorem~1.1 of Loughran and Trepalin~\cite{LoughranTrepalin20} gives a smooth cubic surface $\mathcal X$ of Frobenius class $C_{12}$ over every $\mathbb F_q$. The point-count formula and Table~7.1 of Banwait, Fit\'e and Loughran~\cite{BanwaitFiteLoughran19} give $| \mathcal X(\mathbb F_q)|=q^2+2q+1$. The same table gives Frobenius orbits of lengths $3,6,6,6,6$ on its 27 geometric lines, so none of these lines is defined over $\mathbb F_q$. Restricting the defining homogeneous cubic to any $\mathbb F_q$-line therefore gives a non-zero homogeneous binary cubic, which has at most three projective zeros. Thus $ \mathcal X(\mathbb F_q)\subseteq\operatorname{PG}(3,q)$ has no four collinear points, proving~\eqref{eq:H4-cubic-cap}.
 }%

In fact, if we had a set $ \mathcal X\subseteq \operatorname{PG}(3,q)$ proving $m(q)\ge (1+\varepsilon)q^2$ for some fixed $\varepsilon>0$ and infinitely many $q$, then by averaging over the planes of $\operatorname{PG}(3,q)$ we would get a plane section of $ \mathcal X$ of size at least $(1+\varepsilon-o(1))q$, obviously still
without any four points on a line. This would settle an open problem of Blokhuis for whose resolution he offered 10,000 Hungarian forints (see e.g.~\cite[Section 5, Problem 2]{BallHirschfeld05}).  Thus a constant factor improvement  of~\eqref{eq:H4r} through larger generalised caps seems to require a substantial advance in the underlying finite geometry.


\section*{Acknowledgement}
\noindent 
Oleg Pikhurko was supported 
by ERC Advanced Grant 101020255 and by a grant of access to OpenAI models through the ChatGPT for Academic Researchers program.
Jiabao Yang was supported by National Key R\&D Program of China under grant number 2024YFA1013900, 
NSFC under grant number 12471327, the China Postdoctoral Science Foundation under grant number 2026M793375.
Xiutao Zhu was supported by NSFC under grant number 12401454, Basic Research Program of Jiangsu Province (BK20241361).

The authors thank Jeff Pickhardt for drawing their attention to~\cite{Pickhardt26} and for helpful discussions.

\section*{Declaration}
During the preparation of this work, the authors used AI systems (ChatGPT 5.6 and 6) for various tasks, including generation of candidate
proofs. All AI
suggestions that were incorporated into the final version were verified and streamlined by the authors, who also tried to find closest matches in the literature to the key AI-generated mathematical ideas used in this paper (the cap construction in Theorem~\ref{thm:mainPoisson} and the two checksums in Theorem~\ref{end:thm:augmentations}). The authors take full responsibility for the correctness of the paper.

\smallskip
\noindent\textbf{Data availability.} No data were used for the research described in this article.

\bibliographystyle{alpha}
\bibliography{bibexport}

\end{document}